\documentclass[12pt]{amsart}

\usepackage{gillman_macros}
\usepackage{relsize}
\usepackage{mathrsfs}
\usepackage{multirow}

\newcommand{\proj}{\mathrm{proj}}

\DeclareMathOperator{\diam}{diam}
\DeclarePairedDelimiter{\ceil}{\lceil}{\rceil}

\RequirePackage[colorlinks, linkcolor=blue, citecolor=blue, urlcolor=blue]{hyperref}

\RequirePackage[margin=1.35in]{geometry}

\theoremstyle{plain}

\newtheorem{theorem}{Theorem}[section]
\newtheorem{lemma}[theorem]{Lemma}

\newtheorem{corollary}[theorem]{Corollary}

\theoremstyle{definition}
\newtheorem{definition}[theorem]{Definition}

\newtheorem*{remarks*}{Remarks}

\newtheorem{claim}[theorem]{Claim}

\theoremstyle{definition}
\newtheorem{remark}[theorem]{Remark}
\newtheorem*{remark*}{Remark}

\usepackage{mathtools}

\usepackage{enumerate}

\usepackage[normalem]{ulem}

\numberwithin{equation}{section}

\begin{document}

\title{Large Sets with Small Injective Projections}%Subsets of $\mathbb R^d$ with Small Injective Projections and Maximum Hausdorff Dimension}

\author[Coen]{Frank Coen}
\address{Department of Mathematics \& Statistics, Villanova University,
800 Lancaster Ave, Villanova, PA 19085, U.S.A.}
\email{fcoen@villanova.edu}

\author[Gillman]{Nate Gillman}
\address{Department of Mathematics, Brown University, Providence, Rhode Island 02912, U.S.A.}
\email{ngillman@brown.edu}

\author[Keleti]{Tam\'as Keleti}
\address{Institute of Mathematics, E\"otv\"os Lor\'and University, P\'azm\'any P\'eter S\'et\'any 1/c, H-1117 Budapest, Hungary}
\email{tamas.keleti@gmail.com}
\thanks{The third author was supported by the Hungarian National Research, Development ad Innovation Office - NKFIH, 124749 and 129335.}

\author[King]{Dylan King}
\address{Department of Mathematics \& Statistics, Wake Forest University, Winston-salem, NC 27109, U.S.A.}
\email{kingda16@wfu.edu}

\author[Zhu]{Jennifer Zhu}
\address{Department of Mathematics, University of California, Berkeley, Berkeley, CA 94720, U.S.A.}
\email{jzhu42@gmail.com}

\date {\today}

\begin{abstract}
Let $\ell_1,\ell_2,\dots$ be a countable collection of lines in $\R^d$.
For any $t \in [0,1]$ we construct a compact set $\Gamma\sub\R^d$ with Hausdorff dimension $d-1+t$ which projects injectively into each $\ell_i$, such that the image of each projection has dimension $t$. This immediately implies the existence of homeomorphisms between certain Cantor-type sets whose graphs have large dimensions.
As an application, we construct a collection $E$ of disjoint, non-parallel $k$-planes in $\R^d$, for $d \geq k+2$,
whose union is a small subset of $\mathbb{R}^d$, either in Hausdorff dimension or Lebesgue measure, while $E$ itself has large dimension. As a second application, for any countable collection of vertical lines $w_i$ in the plane we construct a 
%compact 
%\fix{I deleted the word "compact" because I find it  too technical (for a collection of lines) for the Abstract. -T}
collection of nonvertical lines $H$, so that 
%\fix{I rewrote from this point. -T}
$F$, the union of lines in $H$,
%the union of those lines in $H$ 
has positive Lebesgue measure, but each point of each line $w_i$ is contained in at most one $h\in H$ and, for each $w_i$, 
the Hausdorff dimension of $F\cap w_i$ is zero.
%the union over $H$ of such intersections has Hausdorff dimension zero. 
\end{abstract}
\maketitle

%\tableofcontents

%%%%%%%%%%%%%%%%%%%%%%%%%%%%%%%%%%%%%%%%%%%%%%%%%%%%%%%%%%%%
%%%%%%%%%%%%%%%%%%%%%%%%%%%%%%%%%%%%%%%%%%%%%%%%%%%%%%%%%%%%
%%%%%%%%%%%%%%%%%%%%%%%%%%%%%%%%%%%%%%%%%%%%%%%%%%%%%%%%%%%%
%%%%%%%%%%%%%%%%%%%%%%%%%%%%%%%%%%%%%%%%%%%%%%%%%%%%%%%%%%%%
\section{Introduction and statement of results}

Weierstrass famously constructed a function which is everywhere continuous but nowhere differentiable.
The so-called \ita{Weierstrass function} is defined in his original 1872 paper \cite{KW} as the following Fourier series,
\[f(x)=\sum_{n\geq 0}a^n\cos(b^n\pi x),\]
where $0<a<1$, $b$ is a positive odd integer, and $ab>1+3\pi/2$. 
%\fix{I would delete the next sentence. T}
%Half a century later, Hausdorff introduced the notion of \ita{Hausdorff dimension} as a generalization of dimension to non-integer values.
We know now that the graph of the Weierstrass function has Hausdorff dimension greater than one, which provides some explanation for this pathological function's dearth of differentiability: in particular, one can easily show that differentiable functions have graphs 
%with unit dimension \fix{is unit the best word to use here? -D}.
of Hausdorff dimension 1.
%\fixed{I agree with Dylan and changed "unit dimension".}
%We know now that the graph of the Weierstrass function has Hausdorff dimension greater than one, which provides some explanation for the pathological behavior of $f$.
%Since then, there have been plentiful constructions of continuous real-valued functions whose graphs have Hausdorff dimension larger than one.
%\fix{Added by T:}
It is also well known that that there exist 
continuous functions $f:[0,1]\to\R$ with graph of
Hausdorff dimension $2$.

It turns out that the seemingly pathological behavior of a continuous function with a graph of large dimension is the rule rather than the exception. 
Balka, Darji and Elekes recently showed \cite{balka_hausdorff_2016} that for any compact uncountable metric space $K$, within the space of continuous functions $f:K \to \mathbb{R}$, those with graphs of Hausdorff dimension $\dim K+1$ are prevalent in a measure-theoretic sense.
%\fix{Added by T:}
(In this paper \emph{$\dim$} always denotes 
\emph{Hausdorff dimension}.) 
Intuition might suggest that these graphs rely heavily on local oscillations to increase their Hausdorff dimension, and therefore would not be injective. 
Many of the classical constructions take advantage of this strategy. 
For example, the Weierstrass function fails to be injective in the most spectacular way: it lacks monotonicity on all arbitrarily short intervals. 
This is an example of a \ita{continuous non-injective} map with a large graph.
More recently, Eiderman and Larsen found that it is possible to trade continuity for injectivity: they constructed \cite{Eiderman_Larsen} an \ita{injective non-continuous} function on $[0,1]$ whose graph has Hausdorff dimension $2$.

It is therefore natural to ask whether there exist \ita{injective and continuous} real-valued functions that have large graph dimension.
Such a function cannot rely on local oscillations in the same way as the Weierstrass function:
clearly, if a continuous real-valued injective function is defined on an interval, then it is monotone and necessarily has dimension one.
Hence, such a function must be defined on some carefully chosen set.

In the present paper, we answer this question in the affirmative.
We construct compact sets $K_1,K_2 \subset [0,1]$ of dimension $t$, as well as a homeomorphism $f:K_1\to K_2$ so that $\dim \text{graph}(f)=1+t$, for any desired value of $t \in [0,1]$. This dimension is maximal because $\Gamma$ is contained in the Cartesian product $K_1 \times [0,1]$.
The construction of such a function reduces to assembling a set $\Gamma \subset [0,1]^2$ which projects injectively onto $K_1$ in the domain and $K_2$ in the codomain.
Our method of assembling $\Gamma$ is a modified Venetian blind construction, in which we make extra effort to ensure injectivity of the projections.
This generalizes in many ways: first, the two coordinate axes can be replaced with any pair of (not necessarily orthogonal) lines, and this pair of lines can in turn be replaced with any finite or countable collection of lines. 
It is also natural to consider projections into lines inside the ambient space $\mathbb{R}^d$ rather than $\mathbb{R}^2$. 
This is our main result.

\begin{theorem}\label{thm:linesubspaces}
Let $\mathscr L$ be a finite or countable set of lines in $\mathbb R^d$.
Then for any $t\in [0,1]$, there exists a compact set $\Gamma \subseteq \mathbb [0,1]^d$ with  $\dim_H\Gamma=d-1+t$, such that each orthogonal projection $\pi_\ell:\Gamma\to\ell$ is injective with $\dim \pi_\ell(\Gamma)=t$.

Furthermore, consider each of the following statements:
\bnum
    \item The set $\Gamma$ has positive $(d-1+t)$-capacity and infinite $(d-1+t)$-dimensional Hausdorff measure.
    \item The $t$-dimensional Hausdorff measure of every $\pi_\ell(\Gamma)$ is $0$.
\enum
If $t=0$, then (1) holds; if $t=1$, then (2) holds; and if $t\in(0,1)$, then one can choose either of (1) or (2) to hold.
\end{theorem}

%\begin{remark*}
%At the beginning of Section \ref{sec:dimensioncomp} we discuss why one cannot guarantee both (1) and (2) in Theorem \ref{thm:linesubspaces}.
%\end{remark*}

In $\R^d$ we can consider projections into linear subspaces $w$ of any dimension.
Analogously, we construct large $\Gamma$ %with maximum Hausdorff dimension 
such that the projection $\pi_w:\Gamma \to w$ is injective and has dimension $\dim\pi_w(\Gamma)=\dim w-s$ for any prescribed $s \in [0,1]$. 
In this most generalized form, we once again find an easy upper bound on $\dim \Gamma$: since $\Gamma$ is contained in an isometric image of $\pi_w(\Gamma) \times w^\perp$, we have $\dim \Gamma \leq d-s$.
This maximum possible dimension is precisely the one that we obtain as our first corollary.

\begin{corollary}\label{cor:linearsubspaces}
Fix $d\geq 2$, and let $\mathscr W$ be a finite or countable collection of linear subspaces in $\R^d$ (not necessarily all of the same dimension). Then for any $s \in [0,1]$ there exists a compact set $\Gamma \subseteq \mathbb [0,1]^d$ with  $\dim\Gamma=d-s$, such that each projection $\pi_w:\Gamma\to w$ is injective with $\dim\pi_w(\Gamma)= \dim w-s$.
\end{corollary}

Without the injectivity of the projections, this was proved in Claim 2.4 of \cite{chang_small_2018}.
Next, by applying Theorem \ref{thm:linesubspaces} to the standard basis vectors, we obtain the following corollary on the existence of homeomorphisms whose graphs have large dimension. The correspondence between bijective (specifically, coordinate-wise injective) functions $f$ and sets $\Gamma$ injective onto each coordinate axis is clear. That $f$ is a homeomorphism follows 
%\fix{changed the sentence from here - T.}
easily from the compactness of the graph $\Gamma$.
%The correspondence between injective functions $f$ and injective sets $\Gamma$ is clear; such $f$ is continuous because $\Gamma$ is closed, and therefore a homeomorphism.

\begin{corollary}\label{cor:fxnexistence}
For any $d\geq1$ and $t\in [0,1]$, there exist compact $K,K_1,\dots K_d \subset[0,1]$ with dimension $t$ and a coordinate-wise injective homeomorphism $f:K \to K_{1}\times \dots \times K_{d}$ such that $\dim \mathrm{graph}(f)=d+t$. %As in Theorem \ref{thm:linesubspaces}, 
Further, if $t>0$ then each of $K,K_1,\dots K_d$ has $t$-dimensional Hausdorff measure $0$.
\end{corollary}

Denoting by $A(d,k)$ the set of $k$-planes in $\mathbb{R}^d$, we can place a natural metric on $A(d,k)$ through association with $\mathbb{R}^{(k+1)(d-k)}$. Through this metric one can investigate the relationship between the Hausdorff dimension of a collection $E\subset A(d,k)$ and the size (Lebesgue measure or dimension) of its union $B:=B_E$ in $\mathbb{R}^d$. 
In \cite[Theorem 1.3]{Oberlin2014} Oberlin shows that if $B$ has Lebesgue measure zero then $\dim E \leq (k+1)(d-k)-k$, and provides examples which demonstrate that this is tight. 
Concerning the Hausdorff dimension of $B$, in \cite[Corollary 1.12]{hera_hausdorff_2018} H\'era proves that $\dim B \geq k+\dim E/(k+1)$, and provides examples which are tight in some specific cases.
More concretely, for any $s\in[0,(k+1)(d-k)]$
she constructs a collection of
$k$-planes $E\subset A(d,k)$ with $\dim E=s$ such that
the union of the $k$-planes has the following Hausdorff dimension,
\begin{equation}\label{eq:Herafunction}
    h(k,s):=\begin{cases}
s-k\ceil{\frac{s}{k+1}}+2k & \text{if } \ceil{\frac{s}{k+1}} \geq \frac{k+s}{k+1} \\
k+\ceil{\frac{s}{k+1}} & \text{if } \ceil{\frac{s}{k+1}} \leq \frac{k+s}{k+1}.
\end{cases}
\end{equation}
H\'era also formulates the conjecture that this is
the best construction in the sense that whenever
$E\subset A(d,k)$ with $\dim E=s$ and $B$ is the 
the union of the $k$-planes of $E$ then 
$\dim B\ge h(k,s)$.

The examples furnished by H\'era and Oberlin involve collections of $k$-planes which may intersect one another or are parallel.
Since the objective is minimizing the size of $B$, it is not clear whether these intersections or collections of parallel $k$-planes are an important component of the construction. As an application of Corollary \ref{cor:fxnexistence}, we present constructions corresponding to those in \cite{Oberlin2014} and \cite{hera_hausdorff_2018}, with the additional property that they consist of disjoint, nonparallel $k$-planes. We found in Theorem \ref{thm:linesubspaces} that requiring injectivity of a continuous function will not necessarily reduce the Hausdorff dimension of its graph; here we find an analogous statement, that requiring $k$-planes to be disjoint and non-parallel does not necessarily increase the size of their union.

%so too it is not obvious that intersections of $k$-planes can be neglected in the constructions of Oberlin and H\'era. 
%We demonstrate that in both cases these limitations induce no restrictions on Hausdorff dimension.

\begin{theorem}\label{thm:kplanes}
Let $d,k\in\N$ with $d\geq k+2$.
\begin{enumerate}
\item[(i)] There exists a compact set of disjoint, nonparallel $k$-planes $E \subset A(d,k)$ with $\dim E=(k+1)(d-k)-k$ so that $B$,
the union of $k$-planes in $E$, has Lebesgue measure zero.

\item[(ii)]
For any $s$ which satisfies $0 \leq s \leq \dim A(d,k)=(k+1)(d-k)$, there exists a compact set of disjoint non-parallel $k$-planes $E \subset A(d,k)$ with $\dim E=s$ such that $B$, the union of $k$-planes in $E$, has 
Hausdorff dimension $\dim B\le h(k,s)$ for the function
$h(k,s)$ defined in \eqref{eq:Herafunction}.
%the following dimension, 
%\begin{equation}\label{eq:heracases}
%    \dim B \leq \begin{cases}
%s-k\ceil{\frac{s}{k+1}}+2k & \text{if } \ceil{\frac{s}{k+1}} \geq \frac{k+s}{k+1} \\
%k+\ceil{\frac{s}{k+1}} & \text{if } \ceil{\frac{s}{k+1}} \leq \frac{k+s}{k+1}.
%\end{cases}
%\end{equation}
\end{enumerate}

\end{theorem}

% \begin{remark}
% Part (i) provides constructions with $\dim B$ at most the values attained in \cite{hera_hausdorff_2018}. However, the examples given there involved many large collections of parallel lines, with large intersections between distinct families of parallel $k$-planes. The result in (ii) matches examples given in \cite{Oberlin2014}, but once again has the additional property that it consists of disjoint, nonparallel planes.
% \end{remark}

Note that since any compact set $E$ has a compact
subset of any given dimension less than $\dim E$
we can also get $E$ with smaller than the above prescribed dimension.
This observation, in combination with (i) and the result of \cite{Oberlin2014} that if $B$ has Lebesgue measure zero then $\dim E \leq (k+1)(d-k)-k$, 
gives the immediate corollary that we may exchange any such collection $E$ for another consisting of disjoint, nonparallel planes.

\begin{corollary}\label{cor:kplanemeasure}
Suppose $E \subset A(d,k)$ such that $B$, the union of those $k$-planes in $E$, has Lebesgue measure zero.
Then there exists a compact set $E' \subset A(d,k)$ consisting of disjoint, nonparallel $k$-planes such that $\dim E'=\dim E$, with the property that $B'$, the union of the $k$-planes in $E'$, has Lebesgue measure zero.
\end{corollary}

We now consider one final application of Theorem \ref{thm:linesubspaces}.
It is well known that, for a collection of nonvertical lines in the plane which covers a vertical line, the union must have Hausdorff dimension $2$. 
In fact, this is essentially the same as the classical result of Davies \cite{Davies} which states that every Besicovitch set in the plane must have Hausdorff dimension $2$. 
One can ask what we can say in the opposite situation: if a collection of lines in the plane intersects a vertical line $w$ in a small set, does this imply that the union of the lines is small? 
The answer is clearly in the negative: for example, taking all non-vertical lines through a fixed point of $w$ is a counter-example. 
There are two natural ways to exclude this triviality: we could request the chosen lines to intersect $w$ in distinct points; or alternatively, %\fix{I modified the end of this sentence and the beginning of the following one. -T} 
we can require small intersections not only with $w$ but with more than one vertical line.
%we can require small intersection for more than one vertical lines.
By combining Theorem \ref{thm:linesubspaces} with duality and projection theorems we show that even if we have both requirements it is possible that the intersection with the prescribed vertical lines are very small despite the union of the lines being very large. 
In fact, more generally we can construct a collection of hyperplanes in $\R^d$ with these properties.

\begin{theorem}\label{thm:parallellines}
%Let $H:=\{h_i:i\in\mathbb N\}$ be a countable collection of parallel lines in $\mathbb R^d$. 
%For any $\epsilon>0$, there exists a set $L$ of lines, not parallel to the lines in $H$, such that every point of every $h_i$ intersects at most one $\ell \in L$, the set $\cup_{\ell \in L}l$ has positive Lebesgue measure, and $\dim((\cup_{\ell \in L}\ell)\cap h_i)< \epsilon$ for all $i$.
Let $d\ge 2$ and 
let $w_1,w_2,\dots$ be a countable collection of parallel 
lines in $\R^{d}$. 
Then there exists a compact collection $H$ of hyperplanes in $\R^{d}$, not parallel to the lines $w_i$, such that every point of every $w_i$ intersects at most one $h \in H$, the set $F=\cup_{h \in H}h$ has positive Lebesgue measure, and $\dim(F\cap w_i)=0$ for every 
$w_i$.
\end{theorem}

Our paper is organized as follows.
In Section \ref{sec:applications}, we deduce Corollary \ref{cor:linearsubspaces} and Theorem \ref{thm:parallellines} from Theorem \ref{thm:linesubspaces}.
In Section \ref{sec:kplanes} we prove Theorem \ref{thm:kplanes}, using as a crucial ingredient the homeomorphisms furnished by Corollary \ref{cor:fxnexistence}.
In Section \ref{sec:construction} we construct a suitable set $\Gamma$ towards proving Theorem \ref{thm:linesubspaces}.
There we also prove various geometric lemmas relating to our construction.
Finally, in Section \ref{sec:dimensioncomp} we verify that $\Gamma$ and its projections have the alleged dimensions.

%%%%%%%%%%%%%%%%%%%%%%%%%%%%%%%%%%%%%%%%%%%%%%%%%%%%%%%%%%%%
%%%%%%%%%%%%%%%%%%%%%%%%%%%%%%%%%%%%%%%%%%%%%%%%%%%%%%%%%%%%
%%%%%%%%%%%%%%%%%%%%%%%%%%%%%%%%%%%%%%%%%%%%%%%%%%%%%%%%%%%%
%%%%%%%%%%%%%%%%%%%%%%%%%%%%%%%%%%%%%%%%%%%%%%%%%%%%%%%%%%%%
%%%%%%%%%%%%%%%%%%%%%%%%%%%
%%%%%%%%%%%%%%%%%%%%%%%%%%%
%%%%%%%%%%%%%%%%%%%%%%%%%%%
%%%%%%%%%%%%%%%%%%%%%%%%%%%
\section{Proofs of the direct applications of our main result}\label{sec:applications}

%\fixed{modified title -T; approved by N}

%%%%%%%%%%%%%%%%%%%%%%%%%%%
%%%%%%%%%%%%%%%%%%%%%%%%%%%
%%%%%%%%%%%%%%%%%%%%%%%%%%%
%%%%%%%%%%%%%%%%%%%%%%%%%%%
\subsection{Generalization to higher dimensional subspaces}

%First we will argue that Corollary \ref{cor:linearsubspaces} follows from this proposition, which is the special case of our result for which all the subspaces are lines.

\bpf[Proof of Corollary \ref{cor:linearsubspaces}]
Let $\mathscr L$ be a collection of lines such that for each $w \in \mathscr W$ there is some $\ell_w \in \mathscr L$ such that $\ell_w \subset w$. 
By Theorem \ref{thm:linesubspaces}, there exists a compact set $\Gamma$ of Hausdorff dimension $d-s$
such that $\dim\proj_{\ell}=1-s$ for every $\ell\in \mathscr L$.
Since the projections $\pi_{\ell}:\Gamma\to\ell$ are injective, so %too
are the projections $\pi_w:\Gamma\to w$.
Hence, it suffices to show that $\dim \pi_w(\Gamma)=\dim w-s$. Because $\Gamma$ is contained in some isometric image of $w^\perp\times\pi_w(\Gamma)$, we have $\dim\Gamma\leq d-\dim w+\dim\pi_w(\Gamma)$, which implies $\dim\pi_\w(\Gamma)\geq \dim w-s$.
As for the upper bound, by the inclusion 
$\ell_w\subset w$ we have that $\pi_w(\Gamma)$ is contained in some isometric image of $(w\cap\ell_w^\perp)\times\pi_{\ell_w}(\Gamma)$, which has dimension $\dim w-s$.
\epf

%%%%%%%%%%%%%%%%%%%%%%%%%%%
%%%%%%%%%%%%%%%%%%%%%%%%%%%
%%%%%%%%%%%%%%%%%%%%%%%%%%%
%%%%%%%%%%%%%%%%%%%%%%%%%%%
% \subsection{Existence of homeomorphisms}\label{sec:homeomorphisms}

% %\fixed{needs to be rewritten under new hypotheses}

% Here we prove Corollary \ref{cor:fxnexistence}.
% First we apply Theorem \ref{thm:linesubspaces} in $\mathbb{R}^{d+1}$, with $\mathscr L$ consisting of the $d+1$ coordinate axes.
% The resulting set $\Gamma$ has dimension $d+t$, and its projections embed into each axis with corresponding images $K,K_1,\dots K_{d}$, all of dimension $t$. 
% We then define the following homeomorphism, 
% \[f:K \to K_1 \times\cdots \times K_{d}:x\mapsto(x_1,\dots,x_d),\] 
% where $(x,x_1,\dots,x_d)\in\Gamma$ is the unique point whose projection into the first axis is $x$. 
% This map is well-defined and injective because the projections of $\Gamma$ embed into each axis, and it is surjective because the ranges $K_i$ are defined as the projections of $\Gamma$ onto the coordinate axes.
% Finally it is continuous because it is the uniform limit of continuous functions, as the closed sets $\Gamma_n$ form a nested intersection. 

\subsection{Large union of hyperplanes with small injective sections}
\label{sec:parallellines}
\bpf[Proof of Theorem \ref{thm:parallellines}]
For any $x\in\R^{d-1}$ let $v_x$ denote the ``vertical'' line 
$\{x\}\times\R$ in $\R^{d}$. Without loss of generality we can suppose that the parallel lines 
$w_i$ are vertical; that is, they are of the form $w_i=v_{x_i}$ for some $x_i\in\R^{d-1}$.
For any $(a,b)\in\R^{d-1}\times\R$ let $P_{a,b}$
denote the hyperplane $\{(x,y)\in\R^{d-1}\times\R:y=a\cdot x +b\}$ in
$\R^{d}$,
and for any $A\subset\R^{d}$ let 
$E(A):=\cup_{(a,b)\in A} P_{a,b}$.
Then, we have
$$
E(A)\cap v_x = \{(x, a\cdot x +b)\in\R^{d-1}\times\R : (a,b)\in A\} \qquad (x\in\R^{d-1}),
$$ 
and therefore the map $A \mapsto E(A)\cap v_x$ is a scaled copy of the orthogonal
projection of $A$ to a line in the direction $(x,1)$.
%\fixed{didn't follow the geometry of this A: The idea here is say "I go from $A$ to $E(A)\cap v_x$. Now it is perfectly possile that two hyperplanes in $A$ intersect $v_x$ (for this question we have fixed $x$) in the same point, even if those hyperplanes are different in other planes. Can we exactly characterize when this happens?" To approach this, notice right away that two hyperplanes given by $(a_1,b_1),(a_2,b_2)$ intersect $v_x$ at the same point if and only if $a_1\dot x +b_1 = a_2 \dot x + b_2$. Since our objective is to understand exactly which points show up in $E(A)\cap v_x$, we want to mod out the points in $A$ under this equivalence relation. This is exactly what this porjection does. Spse we have the two hyperplanes noted above projecting to the same point on the line in the direction $(x,1)$. Then They can be written $(a_1,b_1)=(\lambda x,\lambda)+(a_1',b_1'),(a_2,b_2)(\lambda x,\lambda)+(a_2',b_2')$, where $x \dot a_1'+b_1'=0$ since this piece is necessarily orthogonal to $(x,1)$. Similarly $x \dot a_2'+b_2'=0$. Then computing $x \dot a_1 + b_1=x \dot (\lambda x+a_1')+\lambda + b_1=\lambda x+\lambda$ by linearity, and the SAME happens for $a_2,b_2$, so these two hyperplanes really do intersect at the same point on $v_x$ and we are taking the projection that considers all such planes equivalent}.

For each $i$ we let $\ell_i$ be a line in $\R^d$ with direction
$(x_i,1)$ and apply Theorem \ref{thm:linesubspaces} to this collection with $t=0$.
This yields a compact set $\Gamma \subset \R^d$ of
positive $(d-1)$-capacity
%Hausdorff dimension $d-1$ 
such that $\pi_{\ell_i}\big{|}_{\Gamma}$ is injective with $\dim \pi_{\ell_i}(\Gamma)=0$.
Now we take $H:=\{P_{a,b} : (a,b)\in \Gamma\}$ and $F:=\cup_{h\in H}h$.
Then $H$ is a compact collection of $(d-1)$-dimensional hyperplanes in $\R^{d}$, not parallel to the lines $w_i$, and also $F=E(\Gamma)$.
The projection of $\Gamma$ into the line $\ell_i$ in the direction $(x_i,1)$
corresponds to the intersection $F\cap w_i$. 
Since these projections are injective, every point of each $w_i$ is contained in at most one $h \in H$.
It is also clear that $\dim(F\cap w_i)=\dim \pi_{\ell_i}(\Gamma) = 0$ for every $w_i$.
 
It remains to check that $F=\cup_{h\in H} h$ has positive
Lebesgue measure.
By a result of Mattila \cite[Corollary 9.10]{mattila_fractal}, if a set has positive $m$-capacity then its projection to almost every $m$-dimensional subspace has positive Lebesgue measure.
We can apply this with $m=1$ and deduce that the projection of $\Gamma$ to almost every line through the origin has positive Lebesgue measure.
%By the Marstrand-Mattila projection theorem \cite[Corollary 9.4]{mattila_fractal}, for almost every line through the origin the projection of the $d-1+\varepsilon>1$ (since for $d=2$ we have $\varepsilon>0$) Hausdorff-dimensional compact set $\Gamma$ has positive Lebesgue measure. 
Thus almost every vertical slice $v_x \cap F$ has positive measure, so by Fubini, $F$ has
positive Lebesgue measure.
\epf

\section{Disjoint non-parallel $k$-planes}\label{sec:kplanes}

In this section we prove Theorem \ref{thm:kplanes}, which consists of modifications of constructions given in \cite{hera_hausdorff_2018} and \cite{Oberlin2014}. In both cases we present constructions with the same Hausdorff dimension as those previously presented, with the additional property that the $k$-planes used are disjoint and non-parallel (whereas in \cite{hera_hausdorff_2018} and \cite{Oberlin2014} they were not).

As stated in the introduction, $A(d,k)$ denotes the set of $k$-dimensional affine subspaces in $\mathbb{R}^d$. 
We use a matrix formulation of the encoding of $A(d,k)$ used in \cite{Oberlin2014}. 
Given a pair $(Y,y_0)$, where $Y$ is a $(d-k)\times k$ matrix and $y_0$ is a $(d-k)\times 1$ vector, we define the following $k$-plane,
% \begin{equation}\label{eq:P(y)defn}
%     P(y,y_0):=\set{\paren{x_1,\dots,x_k,
% y_0^1+\sum_{i=1}^kx_iy_{i,1},\dots,
% %y_0^2+\sum_{i=1}^kx_iy_i^2,\dots,
% y_0^{d-k}+\sum_{i=1}^kx_iy_{i,d-k}}: x \in \mathbb{R}^k}.
% \end{equation}

\begin{equation}\label{eq:P(y)defn}
    P(Y,y_0):=\set{(x,y_0+Y\cdot x): x \in \mathbb{R}^k}.
\end{equation}
% This is simply the set of translates $(x,y_0+x\cdot y)$ where $x\in\R^k$.

%is parameterized via the matrix $y=(y_i ^j)$ for $i \in \{0,\dots, k\}$ and $j \in \{1,\dots, d-k\}$.

Note that this encoding cannot represent all $k$-planes: if a $k$-plane does not pass through a point where the first $k$ coordinates are $0$, then it cannot be encoded in this form. 
For example, in $\mathbb{R}^2$, lines parallel to the $y$ axis cannot be written as $y=mx+b$. However, since this restriction is very weak, almost every plane in $A(d,k)$ can be represented in this way and this is sufficient for our considerations. 
Having encoded almost all elements of $A(d,k)$ as points in $\mathbb{R}^{(k+1)(d-k)}$, we inherit a metric on these $k$-planes from the Euclidean metric on $\mathbb{R}^{(k+1)(d-k)}$.

The proofs of the two parts of Theorem \ref{thm:kplanes} are similarly structured. They seek to create a collection of disjoint, nonparallel $k$-planes $E$ so that $E$ is large, yet the union of those planes found in $E$ is small. This is accomplished by utilizing the function $f$ furnished by Corollary \ref{cor:fxnexistence}. In Equation \eqref{eq:P(y)defn} one may interpret $Y$ as the orientation and $y_0$ the displacement of the given $k$-plane. In our proof we will determine $Y$ and the first $(d-k)-1$ coordinates of $y_0$ by applying $f$ to the $(d-k)$th coordinate of $y_0$. The large dimension of the graph of $f$ will provide the largeness of $E$ while the injectivitiy of $f$ will ensure that such a family of lines is not parallel.

\subsection{$B$ has Lebesgue measure zero}\label{sec:measuresec}
%As it is simpler in technicality, %but identical in spirit, 
%we first prove statement (ii) of Theorem \ref{thm:kplanes}. L %.by constructing $E \subset A(d,k)$ so that $B$, the union of $k$-planes in $E$, has Lebesgue measure zero and $\dim E=(k+1)(d-k)-k$ consists of disjoint, nonparallel $k$-planes.

\begin{proof}[Proof of Theorem \ref{thm:kplanes}, part (i)]
Let $\lambda_d$ denote the $d$-dimensional Lebesgue measure.
By Corollary \ref{cor:fxnexistence}, there exists a compact set $K \subset[0,1]$ with dimension 1 and Lebesgue measure 0, as well as a continuous entry-wise injective function $f: K \to \mathbb{R}^{(k+1)(d-1-k)}$ such that $\dim \text{graph}(f) = (k+1)(d-1-k)+1$. We view the codomain $\mathbb{R}^{(k+1)(d-1-k)}$ as the space of pairs of $(d-1-k)\times k$ and $(d-1-k)\times 1$ matrices over $\mathbb{R}$, by splitting $f$ into $f_1:K \to \mathbb{R}^{(d-1-k)\times k}$ and $f_2:K \to \mathbb{R}^{(d-1-k)\times 1}$. Then we define the following collection of $k$-planes.

%$$E:=\{P(y):y_0^{d-k} \in K,y_1^{d-k}=\dots=y_k^{d-k}=0,(y_i^1,\dots, y_i^{d-k-1}) = f(y_0^{d-k})  \},$$

% \[E:=\{P(y):y=
% \left(\begin{array}{c|c}
%   f(y_0^{d-k})
%   &
%   \begin{matrix}
%   y_0^{d-k}\\
%   0\\
%   \vdots\\
%   0
%   \end{matrix}
% \end{array}\right), y_0^{d-k}\in K)
%   \},\]

\[E:=\set{P(Y,y_0):Y=
\left[\begin{array}{c}
  f_1(t)\\\hline
  0\ \cdots\ 0
\end{array}\right],y_0=\left[\begin{array}{c}
  f_2(t)\\  \hline
  t
\end{array}\right], y_0^{d-k}=t\in K
  },\]
where $P(Y,y_0)$ is defined in (\ref{eq:P(y)defn}).

The function $f(t)$ determines the orientation and positioning of a single $k$-plane lying in $\mathbb{R}^{d-1} \times \{t\}$ for a given $t \in K$. Then $B \subset \mathbb{R}^{d-1} \times K$, and therefore this set satisfies $\lambda_d(B) \leq \lambda_d(\mathbb{R}^{d-1} \times K)=0$, where the last equality is furnished by $\lambda_1(K)=0$.
Furthermore, our representation of $E \subset A(d,k)$ is simply $\text{graph}(f) \times \{0\}\subset A(n,k)$, viewing elements of $A(n,k)$ by their identification in $\R^{(k+1)(d-k)}$. 
Then $\dim E = \dim \text{graph}(f) =(k+1)(d-k-1)+1 = (k+1)(d-k)-k$, as needed. 
Additionally, each $k$-plane in $E$ is disjoint since each $k$-plane is contained within a different slice $\mathbb{R}^{d-1}\times \{t\}$. Since $f$ is injective in each coordinate, each of the $k$-planes will have a different value for $Y_{1,1}$ in particular. Since this coordinate is one component of the orientation of the $k$-planes, they will be nonparallel.
\end{proof}

\subsection{$B$ has limited Hausdorff dimension}\label{sec:dimsec}

\begin{proof}[Proof of Theorem \ref{thm:kplanes}, part (ii)]
 
We modify the construction given in \cite{hera_hausdorff_2018} to select only $k$-planes which are disjoint and nonparallel.
Set $m = \ceil{s/(k+1)}$. If $m=0$ then $s=0$ and setting $E$ to a single $k$-plane suffices. If $m=1\geq (k+s)/(k+1)$, then $s\leq 1$ and so by \cite{hera_hausdorff_2017} taking $E$ any $s$-dimensional collection of disjoint, nonparallel $k$-planes produces $\dim B = k+s$.

%Otherwise, 
%\fix{I deleted "Otherwise" because it might suggest "in the remaining cases". -T}
If $m \geq (k+s)/(k+1)$ and $m\geq 2$, then using Corollary \ref{cor:fxnexistence} we choose some $A \subset [0,1]$ with $\dim A = s-(k+1)(m-1)\in(0,1]$, as well as a coordinate-wise injective homeomorphism $f: A \to \mathbb{R}^{(k+1)(m-1)}$ with $\dim \text{graph}(f) = (k+1)(m-1)+\dim A=s$. 
Once again we view the codomain $\mathbb{R}^{(k+1)(m-1)}$ as the space of pairs of $(m-1)\times k$ and $(m-1)\times 1$ matrices over $\mathbb{R}$, by splitting $f$ into two maps $f_1:K \to \mathbb{R}^{(m-1)\times k}$ and $f_2:K \to \mathbb{R}^{(m-1)\times 1}$.
Then we define the following collection of $k$-planes,
\begin{equation}\label{eqn:defn_of_E}
E:=\set{P(Y,y_0):Y= \left[\begin{array}{c}
\\
  f_1(t)\\
  \\
  \hline
  0 \cdots 0\\
  \hline
  0 \cdots 0\\
  \vdots \ddots \vdots\\
  0 \cdots 0\\
\end{array}\right],y_0=\left[\begin{array}{c}
\\
  f_2(t)\\
  \\
  \hline
  t\\
  \hline
  0\\
  \vdots\\
  0\\
\end{array}\right], y_0^m=t \in A}.
\end{equation}

% \[E:=\set{P(Y,y_0):Y= \left(\begin{array}{c|c}
%   f_1(y_0^{m})
%   &
%   \begin{matrix}
%   0&\cdots&0\\
%   \vdots&\ddots&\vdots\\
%   0&\cdots&0
%   \end{matrix}
% \end{array}\right),y_0=\left(\begin{array}{c|cc}
%   f_2(y_0^{m})
%   &
%   \begin{matrix}
%   y_0^{m}\\
%   \end{matrix}
%   &
%   \begin{matrix}
%   0&\cdots&0\\
%   \end{matrix}
% \end{array}\right), y_0^m \in A}.\]

% \[E:=\{P(y):y=\begin{bmatrix}
%      &  & y_0^m \in A & 0 & \dots  & 0 \\
%      & f(y_0^m) & 0 & \vdots &  & \vdots \\
%      &  & \vdots &\vdots &  & \vdots \\
%      & & 0 & 0& \dots   & 0
% \end{bmatrix}  \}\]
%\\&(y_i^1,\dots y_i^{m-1}) = \left. \begin{cases}f(y_0^{m}) &i < d-k-m\\0 &i\geq d-k-m \end{cases} \right\}
%after the ampersand:
%for all $i$, $(y_i^1,\dots y_i^{m-1})=f(y_0^{m})$, $(y_i^m,\dots y_i^{d-k})=0$
%\begin{cases} f(y_0^{m}),y_{i}^{j}=0 \\ \text{ for }d-k-m \leq j \\ \end{ca\}s. 
%  $$E:=y=(y_i^j):= \begin{cases}
%         y_0^{m}  \in A\\
%         (y_i^j: 1 \leq j \leq m-1) = f(y_0^{m})\\
%         y_i^{m} =0 \text{ if } 1 \leq i \leq k\\
%         y_i^j = 0 \text{ if } d-k-m \leq j \leq d-k.
%      \end{cases}$$
In this case, viewing elements of $A(n,k)$ by their identification in $\R^{(k+1)(d-k)}$, we have $E = \text{ graph}(f) \times \{0\}\subset A(d,k)$, which implies 
$\dim E= \dim \text{graph}(f) = s$, as needed. 
Further, since $B$ is contained within $\mathbb{R}^{m+k-1}\times A$, we also have that $\dim  B \leq m+k-1+\dim A=s-k\ceil{s/(k+1)}+2k$. 
The $k$-planes are disjoint because, as before, they each lie in a different copy of $\mathbb{R}^{m+k-1}$, and they are nonparallel because $f_1$ is coordinate-wise injective. 

Finally, if $m \leq (k+s)/(k+1)$, we again use Corollary \ref{cor:fxnexistence} to choose some $A \subset [0,1]$ with $\dim A = 0$, as well as a coordinate-wise injective homeomorphism $f: A \to \mathbb{R}^{(k+1)m}$ with $\dim \text{graph}(f) = (k+1)m$. Then setting $E'$ as we defined $E$ in equation \eqref{eqn:defn_of_E} above (replacing $m$ with $m+1$ in the definition of $E'$), we have $\dim E' = (k+1)m \geq s$, while $B$ is contained within $\mathbb{R}^{m+k}\times A$.
This implies $\dim B \leq m+k+\dim A=\ceil{s/(k+1)}+k$, as needed. 
Finally, since $E'$ is closed we may take a compact $s$-dimensional subset $E$ of $E'$ to complete the proof.
\end{proof}
 
\begin{remark}

While both of these constructions are at least as strong as the best existing results, (i) is more complete than (ii) because, as it was mentioned in the introduction, there are still gaps in our understanding of the dimension case, regardless of whether the $k$-planes are required to be disjoint or nonparallel. 
%Lower bounds are given on $\dim B$ in \cite{hera_hausdorff_2018}, but the constructions provided, with dimension given by equality in equation (\ref{eq:heracases}), are not tight in all cases. 

With some extra effort we can guarantee 
$\dim B= h(k,s)$ in Theorem~\ref{thm:kplanes} (ii)
%equality in (\ref{eq:heracases}) for our construction  Theorem 1.4 (ii) 
by augmenting $E$ with a suitably chosen simple collection of disjoint non-parallel $k$-planes; it is not difficult to increase $\dim B$ leaving $\dim E$ the same. However, this may not be interesting, since if one happens to get 
$\dim B < h(k,s)$ in Theorem~\ref{thm:kplanes} (ii)
%strict inequality in (\ref{eq:heracases}) 
then this construction surpasses the current best known (even without the extra condition that the $k$-planes are disjoint and non-parallel). In fact, it would give a counter-example to the alread mentioned conjecture of H\'era (\cite[Conjecture 1.16]{hera_hausdorff_2018}), which states that such example cannot exist. In other words, the conjecture of H\'era would imply 
$\dim B= h(k,s)$ in Theorem~\ref{thm:kplanes} (ii).
%that have equality in (\ref{eq:heracases}).

On the other hand, in \cite{Oberlin2014} it is shown that if $B$ has Lebesgue measure zero then $\dim E\leq (k+1)(d-k)-k$, and therefore (i) of Theorem \ref{thm:kplanes} constructs an extremal example. This dichotomy explains why we have Corollary \ref{cor:kplanemeasure} for (i) and not (ii) of Theorem \ref{thm:kplanes}.
\end{remark}

%%%%%%%%%%%%%%%%%%%%%%%%%%%%%%%%%%%%%%%%%%%%%%%%%%%%%%%%%%%%
%%%%%%%%%%%%%%%%%%%%%%%%%%%%%%%%%%%%%%%%%%%%%%%%%%%%%%%%%%%%
%%%%%%%%%%%%%%%%%%%%%%%%%%%%%%%%%%%%%%%%%%%%%%%%%%%%%%%%%%%%
%%%%%%%%%%%%%%%%%%%%%%%%%%%%%%%%%%%%%%%%%%%%%%%%%%%%%%%%%%%%
%%%%%%%%%%%%%%%%%%%%%%%%%%%
%%%%%%%%%%%%%%%%%%%%%%%%%%%
%%%%%%%%%%%%%%%%%%%%%%%%%%%
%%%%%%%%%%%%%%%%%%%%%%%%%%%
\section{The set $\Gamma$}\label{sec:construction}

%\fix{I (Nate) am in the process of making major changes to this subsection.
%I'm not marking any of them with a fix environment.}≠

%In present and subsequent section, we will show that $\Gamma$ suffices to prove Theorem \ref{thm:1}.

%In this first subsection, we construct a compact set $\Gamma\sub\R^d$.

In this section, we construct $\Gamma$ and compute salient attributes of it that will affect dimension and measure computations in the following section.
%Our first step is to show the existence of two sequences $a_k$,$n_k$, entirely independently of geometry. Then we will build $\Gamma$, and finally prove the necessary estimates.

\subsection{Modification and extension of the collection of lines}
%An existence theorem for a sequence of lines}

As we will 
see later, it is prudent to replace our collection of lines $\sL$
%$\ell_1,\ell_2,\dots$ 
with a sequence satisfying a convenient collection of properties.

%\begin{lemma}\label{lem:better_sequence_of_lines}
%Let $\sL:=\{\ell_1,\ell_2,\dots\}$ be a countable collection of lines in $\mathbb R^d$ which go through the origin.
%Then, there exists a sequence of lines $\sL'=\{\ell_1',\ell_2',\dots\}$ satisfying the following:
%\begin{enumerate}
%    \item Every $\ell_i$ appears in $\sL'$ infinitely many times.
%    \item Any $d$ consecutive lines in $\sL'$ have linearly independent directions.
%    \item For every $k$, the lines $\ell_k'$ and $\ell_{k+d}'$ are not orthogonal.
%\end{enumerate}
%\end{lemma}

\begin{lemma}\label{lem:better_sequence_of_lines}
Let $\sL$ be a countable collection of lines in $\mathbb R^d$ which go through the origin.
Then there exists a sequence
%one can list the elements of $\sL$ as 
$\ell_1,\ell_2,\ell_3,\dots$ so that:
\begin{enumerate}
    \item \label{item:infinite_lines}Every $\ell\in\sL$ appears in $\set{\ell_i}$ infinitely many times.
    \item Any $d$ consecutive lines in $\set{\ell_i}$ have linearly independent directions.
    %\item For every $k$, the lines $\ell_k$ and $\ell_{k+d}$ are not orthogonal.
\end{enumerate}
\end{lemma}
\begin{proof}
First, we take $H$ a $d-1$-dimensional subspace in $\mathbb{R}^d$ which does not contain any $\ell \in \mathscr L$, and let $e_1,\dots e_{d-1}$ be lines in $H$ through the origin with linearly independent directions.
Then enumerate the lines $\ell \in \mathscr L$ so that each appears infinitely often, and insert between each line the $d-1$ lines $e_1,\dots,e_{d-1}$. This new enumeration satisfies our constraints.
\end{proof}

\subsection{The construction of $\Gamma$}

Here we construct a compact set $\Gamma\sub\R^d$ which, as we will argue in this section and the next, suffices to prove Theorem \ref{thm:linesubspaces}. It will depend on our choice of two sequences, $n_k$ and $a_k$, which we will specify in Lemma \ref{lem:the_BETTER_sequences_exist}, but for now we define $\Gamma$ for arbitrary positive real sequences $n_k$ and $a_k$.

\begin{definition}\label{def:long_winded_Gamma}
Let $(a_k)$ and $(n_k)$ be positive real sequences.
For $h,j\in\Z$, we define the following interval on the line $\ell_k$:
\begin{align}\label{eqn:Idefn}
    I_{k}^{(h,j)}
    &:=[h\!\cdot\! 2^{-n_{k}+a_{k}}+j\!\cdot\! 2^{-n_{k}+1},
    h\!\cdot\! 2^{-n_{k}+a_{k}}+j\!\cdot \!2^{-n_{k}+1}+2^{-n_{k}}],
\end{align}
%(Here, we define 
where by 
an interval $[a,b]$ on the line $\ell$ 
we mean the
%to be a 
closed line segment connecting $\smash{a\!\cdot\!\hat\ell}$ to $\smash{b\!\cdot\!\hat\ell}$, where $\smash{\hat\ell}$ is the unit vector in the direction of $\ell$.
%\sout{Note that for fixed $k$ and $j$ these are disjoint segments of the same length.} %\fix{I don't quite understand the purpose of this sentence. Fixing $k$ and either $j$ or $h$ gives a disjoint collection of the same length, but disjointness over both is what we will prove later. Is there a reason we only fix $k$ and $j$ here as opposed to $k$ and $h$? D}
Observe that for fixed $k$, these are segments of the same length, and as we'll prove in Lemma \ref{lem:intervals_are_disjoint}, these segments are disjoint for a suitable choice of $a_k$ and $n_k$.
%\fix{More precise. -N}
%Define $m_k'\in\R$ by 
%\[2^{m_k'}:=\sum_{j=d+1}^{k-1}(n_j-a_j-n_{j-d}+\tilde\A_j)+m_d+1.\]
%\fix{To do: replace $m_d$ with upper bound in terms of $n_k,a_k$.}

For a line $\ell_k$, let $\pi_{\ell_k}:\R^d\to\ell_k$ be the orthogonal projection onto $\ell_k$.
We define sets $\Gamma_0\supset\Gamma_1\supset\ldots$ by induction.
We first set 
\[\Gamma_0=\Gamma_1=\cdots=\Gamma_d:=[0,1]^d.\]
Suppose that $\Gamma_{k-1}$ is the union of a collection $R_{k-1}$ of $2^{m_{k-1}}$ identical disjoint solid closed parallelotopes: 
%as follows:
\begin{equation*}
\Gamma_{k-1}=\bigcup_{j=1}^{2^{m_{k-1}}}R_{k-1}^{(j)},
\end{equation*}
%\fix{The problem: in order to define this, we're necessarily unioning over all parallelotopes in $R_{k-1}$. So perhaps we SHOULD, in this definition, IMPLICITLY DEFINE $m_k$ inductively... there's no issue with this, so long as we do it inductively.}
where
%Here we have defined 
$m_{k-1}:= \log_{2}{|R_{k-1}|}$. Using these, we inductively define
\begin{equation}\label{eqn:Gammakdefn}
\Gamma_{k}:=
\bigcup_{h\in\Z}\bigcup_{j=1}^{2^{m_{k-1}}}
%\bigcup_{h,j}
\ 
%^{2^{m_{k}}}\ %^{\frac{-2\sqrt{2}}{2^{-n_k+a_k}}}\
\begin{cases} \pi_{\ell_k}^{-1}(I_{k}^{(h,j)})\cap R_{{k}-1}^{(j)} &\mbox{if this is a parallelotope} \\ 
\emptyset & \mbox{otherwise.}\end{cases}
\end{equation}
%We see that $2^{m_k}$ to be the quantity of parallelotopes that comprise this double union.
Finally, we define
\[\Ga:=\bigcap_{k\ge 1}\Ga_k.\]
\end{definition}

\subsection{Interpreting $\Gamma$}

We now motivate and illustrate this definition.
We defined $\Gamma$ to be the intersection of a nested sequence of compact sets $\Gamma_0\supseteq \Gamma_1\supseteq \Gamma_2\supseteq\dots$, where the $\Gamma_k$ are defined inductively in (\ref{eqn:Gammakdefn}).
Each $\Gamma_k$ is the disjoint union of $2^{m_k}$ identical closed parallelotopes $\smash{R_{k}^{(j')}}$, for $j'=1,\dots,2^{m_k}$; we will use $\smash{R_{k}}$ to denote the collection of such $\smash{R_{k}^{(j')}}$.
We determined the size and relative positioning of these parallelotopes using positive real sequences $(n_k)$ and $(a_k)$, and in this section we will illustrate these geometric objects.
 In the next section we will estimate $m_k$ in terms of these sequences.
 
For the purposes of visualization consider the case when $a_k$ and $n_k$ are both rapidly increasing with $a_k <n_k$. When we assemble $\Gamma_{k}$ from $\Gamma_{k-1}$, from each parallelotope $R_{k-1}^{(j)}$ in $\Gamma_{k-1}$ we are taking many smaller parallelotopes $\smash{R_{k}^{(j')}}$, as in Figure \ref{fig:subpipedsinlargerpipeds}.
\begin{figure}[!ht]
  \centering
  \includegraphics[width=0.8\textwidth]{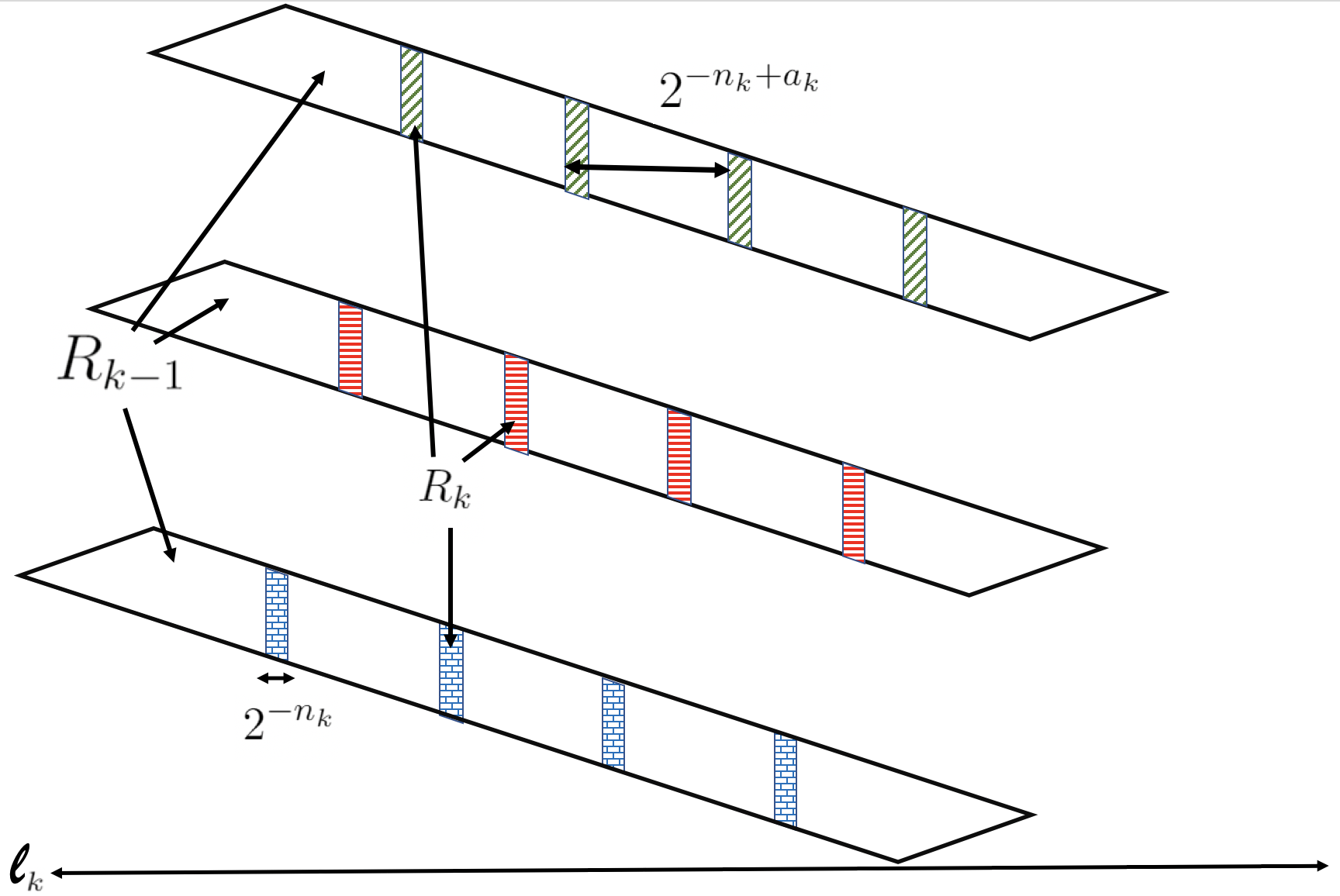}
  \caption{$R_{k}^{(j')}$ inside $R_{k-1}^{(j)}$ and associated distances.}
  \label{fig:subpipedsinlargerpipeds}
\end{figure}
%\sout{In the remainder of this section we omit the subscript $k$, taking $\ell:=\ell_k$.}
Injective projection onto $\ell_k$ is a major desired feature of $\Gamma$.
A natural way to guarantee such injectivity is to require the parallelotopes $\smash{R_{k}^{(j')}}$ %in $\Gamma_{k}$ 
to be contained in preimages, under the projection $\pi_{\ell_k}:\R^d\to\ell_k$, of carefully chosen disjoint intervals in $\ell_k$. These intervals 
%(which we have not yet shown are disjoint) 
were first defined in 
\eqref{eqn:Idefn}
%(\ref{eqn:Gammakdefn}) 
and are each of width $2^{-n_k}$.

To motivate the choice of these intervals, we look ahead to our goal: to bound the Hausdorff dimension of $\Gamma$ from below. 
For this estimate, it will be necessary to place an a lower bound on the distance between two $\smash{R_{k}^{(j')}}$ in $\Gamma_{k}$. 
If, inside a particular $\smash{R_{k-1}^{(j)}}$, we place the new $\smash{R_{k}^{(j')}}$ sufficiently close together, then the distance between the $\smash{R_{k-1}^{(j)}}$ in $\Gamma_{k-1}$ will be very large compared to the distance between $\smash{R_{k}^{(j')}}$ in $\Gamma_{k}$. 
This will ensure that the minimal distance between two parallelotopes in $\Gamma_{k}$ will be achieved only when the pair of polytopes originates from the same parallelotope $\smash{R_{k-1}^{(j)}}$ in $\Gamma_{k-1}$.
This is illustrated in Figure \ref{fig:depictionofintervals}, where the distance between $\smash{R_{k}^{(j')}}$ in different $\smash{R_{k-1}^{(j)}}$ is much larger than the distance between those in the same $\smash{R_{k-1}^{(j)}}$.

Our construction defined an offset of $2^{-n_{k}+a_{k}}$ from the start of one $\smash{R_{k}^{(j')}}$ to the next.
This is a large multiple of the width (measured in distance between opposite faces) of a single parallelotope $\smash{R_{k}^{(j')}}$, so that the distance between two $\smash{R_{k}^{(j')}}$ within the same $\smash{R_{k-1}^{(j)}}$ is at least $2^{-n_{k}+a_{k}}-2^{-n_k}$.
\begin{figure}[!ht]
  \centering
  \includegraphics[width=0.8\textwidth]{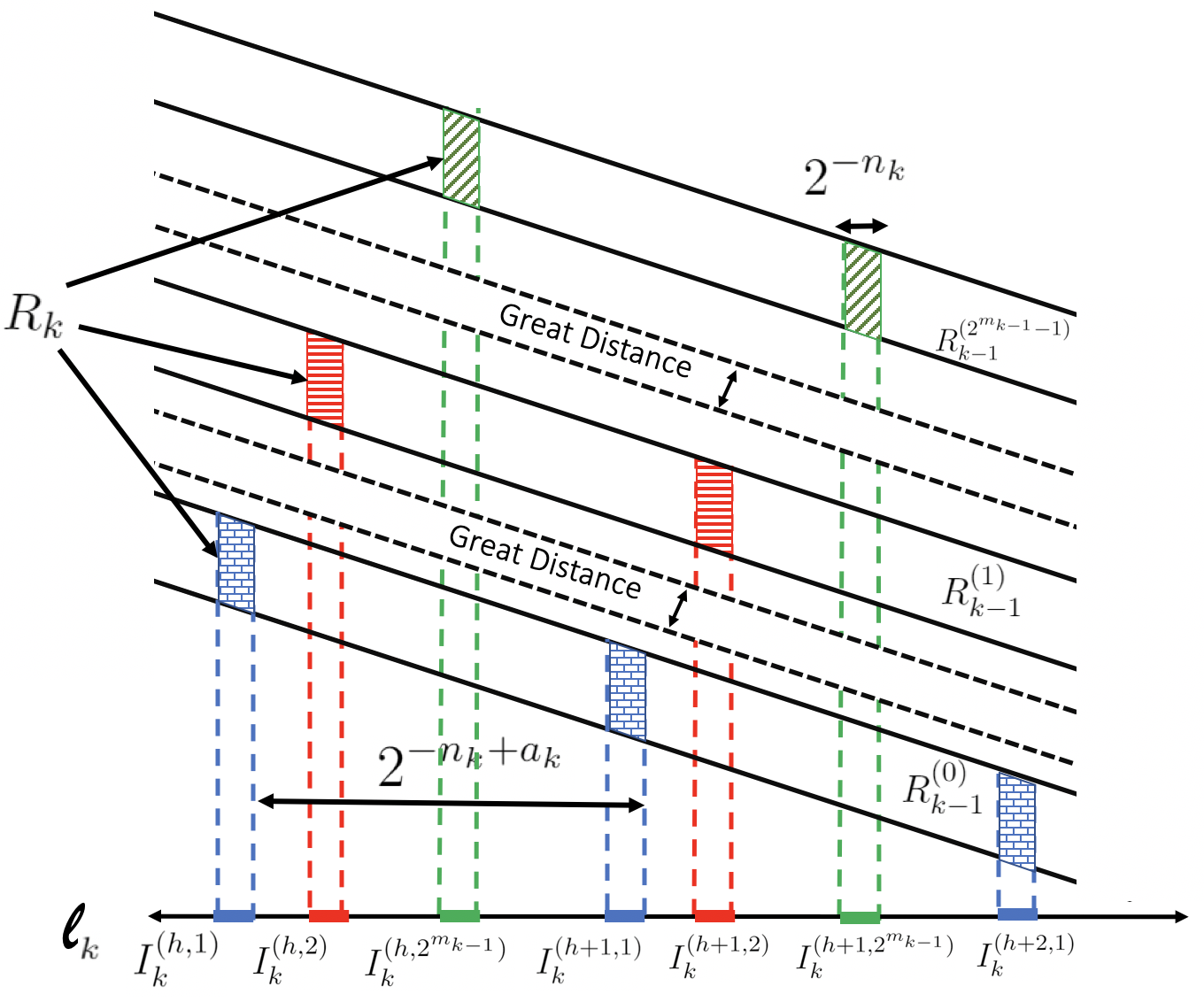}
  \caption{Illustration of the intervals defined in (\ref{eqn:Idefn})}
  \label{fig:depictionofintervals}
\end{figure}
Recall the definition of $I_k^{(h,j)}$ from Equation \eqref{eqn:Idefn}. For fixed $j$, the index $h$ determines an offset of $2^{-n_{k}+a_{k}}$.
In Figure \ref{fig:depictionofintervals}, these intervals are depicted as monochromatic.
Next, we observe that the coefficient on $j$ is small relative to the coefficient on $h$.
Hence, for a fixed $h$ we have that $j$ shifts the interval by a very small distance: in particular, twice the width of a single $\smash{R_{k}^{(j')}}$. Our later constraints on the sequences $a_k$ and $n_k$ will imply that these intervals are disjoint as illustrated. Intervals of the same color in Figure \ref{fig:depictionofintervals} will be disjoint by definition, coming from the same $R_{k-1}^{(j)}$, and we will, in Lemma \ref{lem:the_BETTER_sequences_exist}, force intervals of different colors to be disjoint by choosing $a_k$ large enough. %, as well as a suitable range for the index $j$, 
%\fix{@D: perhaps add something like ``our later constraints on the sequences $a_k$ and $n_k$ will imply that these intervals are disjoint as illustrated'' and perhaps include the ordering here? It just seems slightly unclear right here what the status on disjointness of the intervals is. -N}\fixed{attempted - D}

From a fixed $\smash{R_{k-1}^{(j)}}$ we took as many parallelotopes as this separating distance will allow. The specification that we take only parallelotopes is necessary because it will happen that some $\smash{\pi_{\ell}^{-1}(I_{k}^{(h,j)})}$ intersects the  parallelotope $\smash{R_{k-1}^{(j)}}$ in one of its corners, or more generally any pair of adjacent sides, and in this case the intersection is not a true parallelotope.
In Lemma \ref{lem:mkcomputation}, we show that such discarded sets are negligible so long as we take $n_{k}$ to grow sufficiently fast.

\subsection{Estimating $m_k$}

%\fix{NMG is majorly restructuring this section to make it compatible with Lemma \ref{lem:the_BETTER_sequences_exist}}
%{Geometric lemmas}
\label{sec:geolemmas}

%\fix{Changed the title from "calculating and estimating..." to "estimating..." since "Calculating and" is redundant. also, technically we only ever estimate it. -N}

%\fixed{As far as I see this section is only about $m_k$,so I completely rewrote the title and the first paragraph. -T Looking at it now. -N}
As is apparent from Definition \ref{def:long_winded_Gamma}, the construction of $\Gamma$ is completely determined by 
the sequences $(n_k),(a_k),$ and $(\ell_k)$. 
In particular, in order to calculate
the size of $\Gamma$ and its projections, we 
need good estimates on $m_k$ in terms of the 
given sequences $(n_k)$ and $(a_k)$.

Consider the projection of $R_{k-1}^{(j)}$ to $\ell_k$, and recall that in (\ref{eqn:Gammakdefn}) we must discard those sets where the preimage of this projection is in a ``corner'' of $\smash{R_{k-1}^{(j)}}$.
%\fix{New from here. T}
In other words, we would like to estimate the length of the 
interval $I$ for which 
$\pi_{\ell_k}^{-1}(I_{k}^{(h,j)})\cap R_{{k}-1}^{(j)}$
is a parallelotope if and only if $I_{k}^{(h,j)}\subset I$.
The following lemma gives the estimate we need.

\begin{lemma}\label{l:intervallength}
Let $(\ell_k)$ be the sequence of lines given by Lemma \ref{lem:better_sequence_of_lines}, 
%and let $(m_k)$ be the corresponding real sequence provided in Definition \ref{def:long_winded_Gamma}.
let $(a_k)$ and $(n_k)$ be positive real sequences, and let $\Gamma=\Gamma\big((\ell_k),(a_k),(n_k)\big)$ be as in Definition \ref{def:long_winded_Gamma}.

There exist real numbers $\alpha_{d+1},\alpha_{d+2},\ldots$
and $\beta_1,\beta_2,\ldots$ that depend only on the sequence
$\ell_1, \ell_2,\ldots$ such that 
for any $k>d$, under the assumption
\begin{equation} \label{e:assumption}
n_i \ge n_{k-d}+\alpha_k + \beta_i 
\qquad
(i=k-d+1,\ldots,k-1),
%    \sum_{i=k-d+1}^{k-1} 2^{-n_i} \beta_i \le
%    2^{-n_{k-d}-\alpha_k-1}},
\end{equation}
the following holds.

For each $j=1,\ldots,2^{m_{k-1}}$ there 
exists a nonempty interval $I=I(j,k)$ of length
\begin{equation}\label{e:Isize}
2^{-n_{k-d}-\alpha_k-1} \le |I| \le 2^{-n_{k-d}-\alpha_k}   
\end{equation}
such that for every $h\in\Z$ the
set $\pi_{\ell_k}^{-1}(I_{k}^{(h,j)})\cap R_{{k}-1}^{(j)}$
is a parallelotope if and only if $I_{k}^{(h,j)}\subset I$.
\end{lemma}
\begin{proof}
Fix $k$ and $j$ and let $R=R_{{k}-1}^{(j)}$.
Let $I$ be the set of those real numbers $t$ for which 
the hyperplane $\pi_{\ell_k}^{-1}(\{t\})$ is between two 
opposite faces of the parallelotope $R$.
Note that then indeed
\smash{$\pi_{\ell_k}^{-1}(I_{k}^{(h,j)})\cap R$}
is a parallelotope if and only if \smash{$I_{k}^{(h,j)}\subset I$.}
By definition $I$ is nonempty if $R$ has two opposite faces
such that their orthogonal projections to $\ell_k$
are disjoint and clearly the length of $I$ is the distance
between these projections.

Note that by construction $R$ has $d$ pairs of opposite faces
$(F_{k-d},F'_{k-d}),\ldots,$ $(F_{k-1},F'_{k-1})$ such that
for every $i=k-d,\ldots,k-1$ the faces $F_i$ and $F'_i$ are
perpendicular to $\ell_i$ and the distance between the 
hyperplanes containing $F_i$ and $F'_i$ is $2^{-n_i}$.

\begin{figure}[!ht] 
  \centering
  \includegraphics[width=0.8\textwidth]{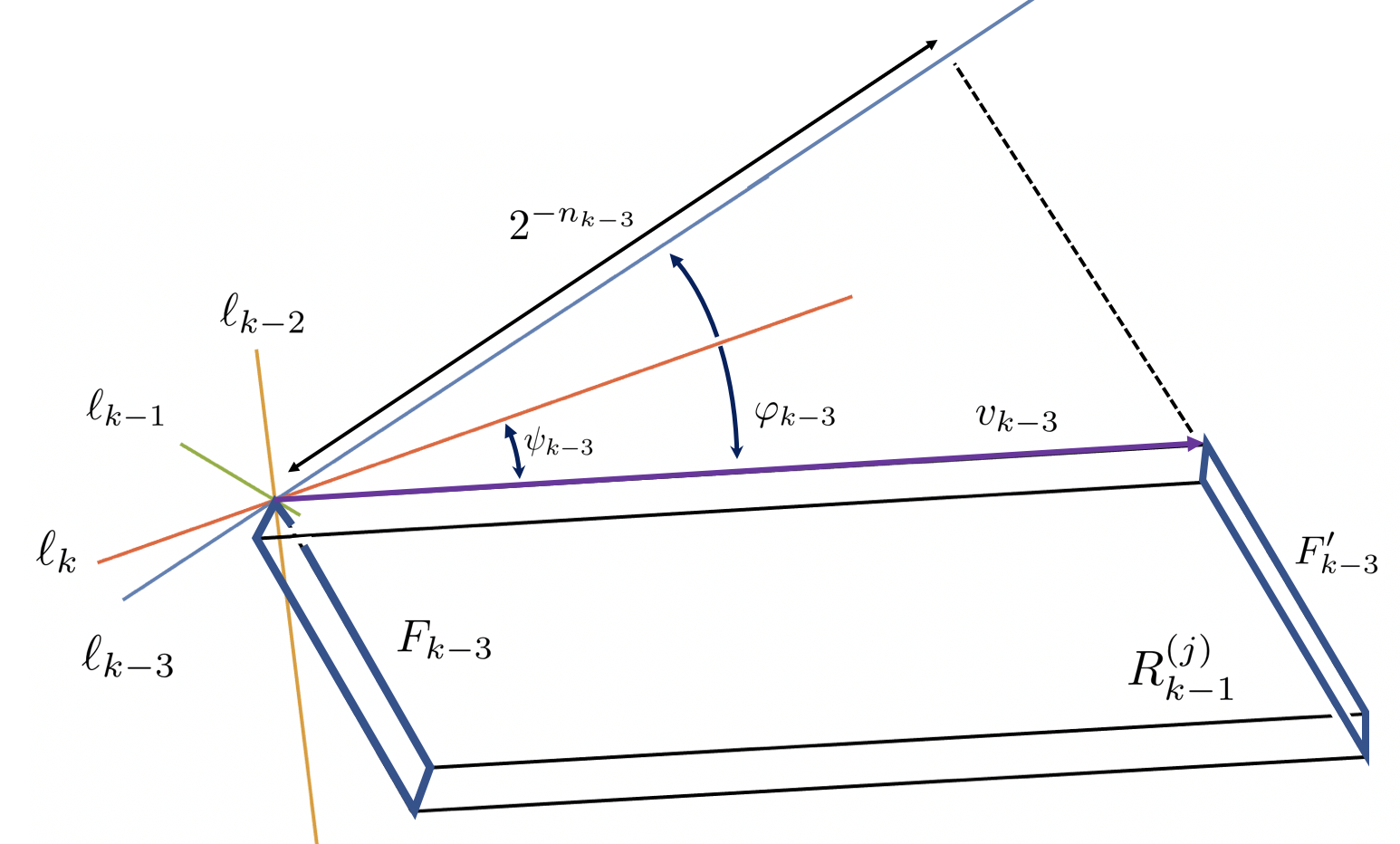}
  \caption{Illustration of the geometric ideas used in Lemma \ref{l:intervallength}. In blue, yellow, and green we see the lines which define the shape of the parallelotope $R_{k-1}^{(j)}$. Each of these lines corresponds to a pair of faces. In the figure the faces corresponding to $\ell_{k-3}$ have been boldened and the vector $v_{k-3}$ and angle $\phi_{k-3}$ drawn. The fourth line $\ell_k$ is the ``target'' onto which we will be projecting in the present step, so the angles $\psi$ inform the projection of $R_{k-1}^{(j)}$ onto this line.}
  \label{fig:3dptope}
\end{figure}

For each $i=k-d,\ldots,k-1$ let $v_i$ be the vector such that
$F_i + v_i = F'_i$. 
Then $v_i$ is parallel to all faces of $R$
but $F_i$ and $F'_i$, which implies that $v_i$ is perpendicular
to every line $\ell_{k-d},\ldots,\ell_{k-1}$ but $\ell_i$.
Let $\phi_i$ be the angle between $v_i$ and $\ell_i$
and let $\psi_i$ be the angle between $v_i$ and $\ell_k$.

Note that 
\begin{equation}\label{e:cosphi}
    |v_i|=\frac{2^{-n_i}}{\cos\phi_i}.
\end{equation}
Since by (2) of Lemma~\ref{lem:better_sequence_of_lines}
the directions of $\ell_{k-d},\ldots,\ell_{k-1}$ are 
linearly independent, the directions of $v_{k-d},\ldots,v_{k-1}$
are determined by $\ell_{k-d},\ldots,\ell_{k-1}$, hence we 
obtain that the angles $\phi_i$ depend only on the sequence of
lines $\ell_1,\ell_2,\ldots$.
We claim that 
$\phi_i\neq \pi/2$ for any $i=k-d,\ldots,k-1$ and
$\psi_{k-d}\neq \pi/2$. 
Indeed,
$\phi_i = \pi/2$ would imply that
$\ell_{k-d},\ldots,\ell_{k-1}$ are all perpendicular to $v_i$
and $\psi_{k-d}= \pi/2$ would imply that
$\ell_{k-d+1},\ldots,\ell_{k-1}, \ell_k$ are all 
perpendicular to $v_{k-d}$, which in both cases
would contradict the linear independence assumption (2) of Lemma~\ref{lem:better_sequence_of_lines}.
The geometric setup is illustrated in Figure \ref{fig:3dptope}.

Since $|I|$ is the distance between $\pi_{\ell_k}(F_{k-d})$
and $\pi_{\ell_k}(F'_{k-d})$, provided this distance is positive,
we get that
\begin{equation}
\label{e:Iestimates}
|\pi_{\ell_k}(v_{k-d})|-\diam(\pi_{\ell_k}(F_{k-d})) 
\le |I| \le
|\pi_{\ell_k}(v_{k-d})|,
\end{equation}
provided the lower estimate is positive.
Define $\alpha_k$ and $\beta_i$ $(i=k-d+1,\ldots,k-1)$ by
\begin{equation}
2^{-\alpha_k}=\frac{\cos\psi_{k-d}}{\cos\phi_{k-d}}
\qquad \textrm{and} \qquad
2^{-\beta_i}=\frac{\cos \phi_i}{2(d-1)}.
%\beta_i=\frac{1}{\cos \phi_i}.
\label{eqn:alpha_beta_defn}
\end{equation}
These numbers are well defined since none of
these angles can be $\pi/2$ and they depend only on the 
sequence of lines
$\ell_1,\ell_2,\ldots$ and their own indices.
Note that 
\[
|\pi_{\ell_k}(v_{k-d})|=|v_{k-d}|\cos\psi_{k-d}=
\frac{2^{-n_{k-d}}}{\cos\phi_{k-d}}\cos\psi_{k-d}=
2^{-n_{k-d}-\alpha_k},
\]
which gives the upper estimate of \eqref{e:Isize} via \eqref{e:Iestimates}.
By this and \eqref{e:Iestimates}, to get the lower estimate 
in \eqref{e:Isize} it is enough to show that
\begin{equation}\label{e:enough}
 \diam(\pi_{\ell_k}(F_{k-d})) \le 2^{-n_{k-d}-\alpha_k-1}.   
\end{equation}
Using that projection cannot increase the distance,
the edges of $F_{k-d}$ have lengths $|v_{k-d+1}|,\ldots,|v_{k-1}|$,
and \eqref{e:cosphi},
%the definition of $\beta_i$
we obtain
\[
\diam(\pi_{\ell_k}(F_{k-d})) \le 
\diam(F_{k-d}) \le
\sum_{i=k-d+1}^{k-1} |v_i| =
\sum_{i=k-d+1}^{k-1}\frac{2^{-n_i}}{\cos\phi_i}. 
% =\sum_{i=k-d+1}^{k-1} {2^{-n_i}}\beta_i.
\]
On the other hand the assumption \eqref{e:assumption}
and the definition of $\beta_i$ gives
\[
\frac{2^{-n_i}}{\cos\phi_i} \le
\frac{2^{-n_{k-d}-\alpha_k-\beta_i}}{\cos\phi_i} =
\frac{2^{n_{k-d}-\alpha_k}}{2(d-1)}.
\]
Combining these 
%this with the assumption \eqref{e:assumption} 
we get \eqref{e:enough},
which completes the proof.
\end{proof}

The following lemma contains our requirements about the 
sequences $(a_k)$ and $(n_k)$ in the construction.

\begin{lemma}\label{lem:the_BETTER_sequences_exist}
Fix $t\in[0,1]$. 
Let $(\ell_k)$ be the sequence of lines given by Lemma \ref{lem:better_sequence_of_lines}, 
and let
$\alpha_{d+1},\alpha_{d+2},\ldots$
and $\beta_1,\beta_2,\ldots$ 
be the sequences given by Lemma~\ref{l:intervallength}.

There exist positive real sequences $(a_k)$ and $(n_k)$ 
such that for every $k>d$,
%the assumption \eqref{e:assumption} of Lemma~\ref{l:intervallength} and 
all of the following conditions hold:
\begin{enumerate}[I.]
    \item\label{item:prev_lemma_assumption}
    $n_i \ge n_{k-d}+\alpha_k + \beta_i \qquad
    (i=k-d+1,\ldots,k-1)$
    \item \label{item:limiting_diam_assumption}
    $n_k \geq \beta_k + k$
    \item\label{item:mk_estimating1} $n_k-a_k\ge n_{k-d}+\A_k+1$
    \item\label{item:dek_ALL_CASES} $n_k \ge 2k\cdot(2n_{k-1}+a_{k-1}+n_{k-d}+|\alpha_k|+2)$
    \item\label{item:ek_ALL_CASES} $n_k\ge 4k\cdot\big(2n_{k-1}+a_{k-1}+n_{k-d}+\A_k+\A_{k+d}+3d+2+\sum_{j=1}^{d-1}(n_{k-d+j}+\A_{k+j})\big)$
    \item\label{item:ek'_condition} $n_k\ge 4k\cdot(\A_{k+d}+2)$
    \item\label{item:inj_condition} $a_k \ge 2n_{k-1}-a_{k-1}+2$
    \item\label{item:option_2} $n_k \ge 2k^2$
    \item\label{item:furnishes_MDP_cases} $n_{k+1}-a_{k+1}\ge n_k$
    \item\label{item:case1_MDP} $(\log_2 n_k)/n_k\le 1/(4k)$
    \item\label{item:limit_ak/nk} $\lim_{k\to\infty}a_k/n_k =1-t$ %\fix{If this stays, must update eqrefs}\fixed{given item label}
    \item\label{item:pm_limit} Consider the following statements: 
        \begin{enumerate}
            \item[(1)] For large enough $k$, we have $1-t\ge \frac{a_k}{n_k}+\frac{1}{k}.$
            \item[(2)] For large enough $k$, we have $1-t\le\frac{a_k}{n_k}-\frac{1}{k}.$
        \end{enumerate}
    If $t=0$, then (1) holds; if $t=1$, then (2) holds; and if $t\in(0,1)$, then we can choose either of (1) or (2) to hold.
\end{enumerate}
%All of the aforementioned growth conditions are compatible with imposing the limiting ratio
%\begin{equation}\label{eq:limit_ak/nk}
%\lim_{k\to\infty}\frac{a_k}{n_k} =1-t.
%\end{equation}

%Furthermore, consider each of the following statements:
%\begin{enumerate}
%    \item For large enough $k$, we have 
%    $1-t\ge \frac{a_k}{n_k}+\frac{1}{k}.$
%    \item For large enough $k$, we have
%    $1-t\le\frac{a_k}{n_k}-\frac{1}{k}.$
%\end{enumerate}
%If $t=0$, then (1) holds; if $t=1$, then (2) holds; and if $t\in(0,1)$, then we can choose either of (1) or (2) to hold.
%Either case is compatible with all the previously stated growth conditions on the sequences $a_k$ and $n_k$\fix{seems repetitive with earlier sentence-D}.
\end{lemma}

\begin{remark}
In Lemma \ref{lem:the_BETTER_sequences_exist}, the choice of imposing growth condition (1) or (2) on $a_k/n_k$ corresponds to the choice of alternative (1) or (2) in Theorem \ref{thm:linesubspaces}.
We will justify this correspondence in Section \ref{sec:dimensioncomp}.
\end{remark}
%\fix{I shortened this proof. -N}
\bpf[Proof of Lemma \ref{lem:the_BETTER_sequences_exist}]
For every $k$, let us initially set 
\begin{equation}\label{eq:initial_seq_defn}
\nu'_k = 1-t\pm \frac{1}{k},\qquad 
\nu_k = \begin{cases}
\nu'_k & \textrm{if }\nu'_k\in(0,1)\\
1/2 & \textrm{otherwise},
\end{cases}
\qquad \textrm{and} \qquad
a_k=\nu_k\cdot n_k,
\end{equation}
%\fix{Unfortunately the above $\nu_k$ can be negative for small $k$. T}
where we choose $+1/k$ if we want (2) in property \ref{item:pm_limit} and $-1/k$ if we want (1) in property \ref{item:pm_limit}. 
%\fix{Shouldn't we define $a_k$ for $k\le f_d(t)$ for completeness? T}
This choice satisfies properties \ref{item:limit_ak/nk} and \ref{item:pm_limit} for any sequence $n_k$. We will now show that there is a choice of $n_k$ which increases rapidly enough to satisfy the remaining properties.
%Observe that each of these three properties is invariant under scaling $n_k$ (or equivalently, scaling the pair $(a_k,n_k)$) by a constant $c>0$, so we may replace $n_k$ with $cn_k$ arbitrarily. %each pair $(a_k,n_k)$ with $(ca_k,cn_k)$ arbitrarily. 
Property \ref{item:case1_MDP} is satisfied 
for any sufficiently large $n_k$, since $\lim_{t\to\infty}t\inv\log_2t=0$, %so in particular it follows from property~\ref{item:option_2}.\fix{Incorrect deduction; for $d=2$ and $n_k=2k^2$, inequality \ref{item:case1_MDP} doesn't hold for $k=3$. -N}
and similarly for property~\ref{item:option_2}.

By algebraically substituting %\fix{deleted 'replacing' (sounds like we are changing) and added 'algebraically substituting D} 
every $a_i$ by $\nu_i\cdot n_i$ and using that
$0<\nu_i<1$, we can see that all properties \ref{item:limiting_diam_assumption}
through \ref{item:furnishes_MDP_cases} are of the form
$c\cdot n_i \ge f(n_1,\ldots,n_{i-1})$, where the
constant $c>0$ and 
the function $f$ depend only on $i,d$ and the sequence of
lines $\ell_1,\ell_2,\ldots$; also, property \ref{item:prev_lemma_assumption} 
consists of finitely many inequalities of this form.
Therefore by induction all of these properties 
can be satisfied.
%
%Property~\ref{item:prev_lemma_assumption} can
%be satisfied by choosing $n_{k-d+1},\ldots, n_{k-1}$ 
%much larger than $n_{k-d}$.
%Properties \ref{item:mk_estimating1} and \ref{item:furnishes_MDP_cases} may be satisfied by setting $n_k$ large enough; 
%    for if $t \ne 0$, then $k>1/t$ implies $a_k < n_k$, so taking $n_k$ (and consequently $a_k$) large enough suffices;
%    and if $t  =  0$, then $a_k<n_k$ for all $k$, so again $n_k$ (and consequently $a_k$) large enough suffices. 
%It is easier to see that properties \ref{item:dek_ALL_CASES} through \ref{item:option_2} hold so long as $n_k$ is large enough.
%Additionally, property \ref{item:inj_condition} holds so long as $a_k$ is large enough, and for this it suffices to take $n_k$ large enough.
%Likewise, property \ref{item:case1_MDP} is satisfied because $\lim_{t\to\infty}t\inv\log_2t=0$.
%In all of these cases, we may adjust $n_k$ at will because we do it inductively; the right hand sides contain (at worst) lower order terms, so they do not change as we set $n_k$ (and consequently $a_k$) very large.
\epf

Now that we can estimate the width of the valuable space inside the projection of $\smash{R_{k-1}^{(j)}}$ to $\ell_k$, %\fix{do we use $\ell$ instead of $\ell_k$ anywhere from here? I don't see D}, 
we estimate the number of intervals that can fit into the projection of a single $\smash{R_{k-1}^{(j)}}$.
This allows us to effectively estimate the quantity of new parallelotopes $\smash{R_{k}^{(j')}}$ born from a single $\smash{R_{k-1}^{(j)}}$, which in turn allows us to estimate the number of parallelotopes $\smash{R_{k}^{(j')}}$ inside $\Ga_k$.

%\fix{7/17: Nate is in the process of majorly restructuring this lemma...}

\begin{lemma}\label{lem:mkcomputation}
Fix $t\in[0,1]$.
Let $(a_k)$ and $(n_k)$ be real sequences given by Lemma \ref{lem:the_BETTER_sequences_exist}, 
let $(\ell_k)$ be the sequence of lines given by Lemma \ref{lem:better_sequence_of_lines}, 
and let $(m_k)$ be the corresponding real sequences provided in Definition \ref{def:long_winded_Gamma}.
For $k>d$, let $\alpha_k$ be as in Lemma \ref{lem:the_BETTER_sequences_exist}.
Then, %there exists $M_k>0$, depending only on the angle between the lines $\ell_{k-d}$ and $\ell_{k}$, such that 
for some $\tilde\A_{k}\in[\alpha_k-1,\alpha_k+2]$, we have
\begin{equation}\label{eq:mkcomputation}
    m_k=n_k-a_k+m_{k-1}-n_{k-d}-\tilde\A_{k}.
\end{equation}
%provided the sequences $\set{n_k}$ and $\set{n_k-a_k}$ grow sufficiently fast.
%\fix{We also need in the proof that $n_k-a_k$ is large enough, and this is not that obvious when $t=0$, so this requirement has to be mentioned in this lemma, and need to be checked around (5.1) and (5.2). -T}

%\fix{Need to rephrase, in light of Lemma \ref{lem:the_sequences_exist}... need to fix ``converging to zero arbitrarily fast''}
In particular, there exist real sequences $(\de_k)_{k\ge d+1}$, $(\e_k)_{k\ge d+1}$, and $(\e_k')_{k\ge d+1}$ with $|\de_k|,|\e_k|,|\e_k'|\le 1/(2k)$ such that the following hold for $k> d$:
\bnum
    \item[i)] We have $m_k=(1+\de_k)n_k-a_k$.
    \item[ii)] We have $-m_{k+d}+\sum_{j=1}^d (n_{k+j}-a_{k+j}) +d = a_k + \e_k n_k$.
    \item[iii)] We have $-m_{k+d-1}+\sum_{j=1}^{d-1}( n_{k+j}-a_{k+j}) +d = a_k + (-1+\e'_k) n_k$.
\enum
%Furthermore, by taking $n_k$ sufficiently large we can require $\de_k,\e_k,$ and $\e_k'$ to limit to zero as quickly as we want.
\end{lemma}
\begin{proof}
Let $k>d$. 
We estimate $m_k$, which we recall is %\sout{uniquely} 
entirely determined %\fix{I do not like unique here. Perhaps 'entirely determined'? D} 
by our inductive definition (\ref{eqn:Gammakdefn}).

Fix $j\in \{1,\ldots,2^{m_{k-1}}\}$ and let 
$r_{k-1}^j$ be the quantity of $\smash{R_k^{(j')}}$ inside $\smash{R_{k-1}^{(j)}}$.
By Property \ref{item:prev_lemma_assumption} of Lemma~\ref{lem:the_BETTER_sequences_exist}, the assumption
\eqref{e:assumption} of Lemma~\ref{l:intervallength}
holds, so the conclusion holds as well. 
Let $I$ be the interval Lemma~\ref{l:intervallength} gives.
By construction, we have
\[
\left\lfloor\frac{|I|}{2^{-n_k+a_k}}\right\rfloor \le 
r_{k-1}^j \le 
\left\lceil\frac{|I|}{2^{-n_k+a_k}}\right\rceil.
\]
Combining this with \eqref{e:Isize} of Lemma~\ref{l:intervallength} we obtain 
\begin{equation*}
% \left\lfloor\frac{2^{-n_{k-d}-\alpha_k-1}}{2^{-n_k+a_k}}\right\rfloor \le 
\left\lfloor 2^{n_k-a_k-n_{k-d}-\alpha_k-1}\right\rfloor \le 
r_{k-1}^j \le 
\left\lceil 2^{n_k-a_k-n_{k-d}-\alpha_k}\right\rceil.
%\left\lceil\frac{2^{-n_{k-d}-\alpha_k}}{2^{-n_k+a_k}}\right\rceil.   
\end{equation*}
Using \eqref{item:mk_estimating1} of Lemma~\ref{lem:the_BETTER_sequences_exist} and
the fact that $x/2\le \lfloor x \rfloor$ 
and $\lceil x \rceil \le 2x$
for any $x\ge 1$, this gives
\begin{equation*}\label{e:r-estimate}
2^{n_k-a_k-n_{k-d}-\alpha_k-2} \le 
r_{k-1}^j \le 
2^{n_k-a_k-n_{k-d}-\alpha_k+1}. 
\end{equation*}
Since by definition 
$2^{m_k}=\sum_{j=1}^{2^{m_{k-1}}}r_{k-1}^j $
this implies that
\[
m_{k-1}+n_k-a_k-n_{k-d}-\alpha_k-2 \le 
m_k \le 
m_{k-1}+n_k-a_k-n_{k-d}-\alpha_k+1, 
\]
which completes the proof of the first paragraph of the lemma.

Now we verify i).
If we define 
\begin{equation}\label{eq:dek_defn}
\de_k:=(m_{k-1}-n_{k-d}-\tilde\A_k)/n_k,
\end{equation}
then an elementary calculation verifies that $m_k=(1+\de_k)n_k-a_k$, so it remains to show that $|\de_k|\le 1/(2k)$.
We show this by induction on $k$.
As (by Definition~\ref{def:long_winded_Gamma}) $m_d=0$, the base case $|\de_{d+1}|\le 1/2(d+1)$ is equivalent to $n_{d+1}\ge 2(d+1)\cdot (n_1+\tilde\A_{d+1})$, which is implied by property \ref{item:dek_ALL_CASES} in Lemma \ref{lem:the_BETTER_sequences_exist} for $k=d+1$.
Now let $k\ge d+2$, and assume that $|\de_{k-1}|\le 1/2(k-1).$ %\le 1$.
By definition, 
%$|\de_k|=|m_{k-1}-n_{k-d}-\tilde\A_k|/n_k$.
$|\de_k|%=|m_{k-1}-n_{k-d}-\tilde\A_k|/n_k
\le 
(m_{k-1}+n_{k-d}+|\tilde\A_k|)/n_k$.
Using i) for $k-1$, $\tilde\A_k\in[\alpha_k-1,\alpha_k+2]$
and $|\delta_{k-1}|<1$, we obtain
\[|\de_k|
%\le\frac{m_{k-1}-n_{k-d}+\A_k+1}{n_k}
\le\frac{(1+\de_{k-1})n_{k-1}-a_{k-1}+n_{k-d}+|\A_k|+2}{n_k}
\le \frac{2n_{k-1}-a_{k-1}+n_{k-d}+|\A_k|+2}{n_k},
\]
which is at most $1/(2k)$ by property \ref{item:dek_ALL_CASES} of Lemma \ref{lem:the_BETTER_sequences_exist}.
%
%If $m_{k-1}-n_{k-d}-\tilde\A_k\ge 0$, then we can estimate
%\[|\de_k|
%%\le\frac{m_{k-1}-n_{k-d}+\A_k+1}{n_k}
%\le\frac{[(1+\de_{k-1})n_{k-1}-a_{k-1}]-n_{k-d}+\A_k+1}{n_k}
%\le \frac{2n_{k-1}-a_{k-1}-n_{k-d}+\A_k+1}{n_k},
%\]
%which is at most $1/2k$ by property \ref{item:dek_ALL_CASES} of Lemma \ref{lem:the_BETTER_sequences_exist}.
%And if $m_{k-1}-n_{k-d}-\tilde\A_k<0$, then we can estimate
%\[|\de_k|
%%\le\frac{-m_{k-1}+n_{k-d}+ \A_k+1}{n_k}
%\le\frac{-[(1+\de_{k-1})n_{k-1}-a_{k-1}]+n_{k-d}+\A_k+1}{n_k}
%\le\frac{a_{k-1}+n_{k-d}+\A_k+1}{n_k},
%\]
%which is at most $1/2k$ by property \ref{item:dek_ALL_CASES} of Lemma \ref{lem:the_BETTER_sequences_exist}.
This proves i).

Next we verify ii).
Towards this, we define
\begin{equation}\label{eq:ek_defn}
\e_k:=\frac{-m_{k-1}+n_{k-d}+\tilde\A_k+\tilde\A_{k+d}+d+\sum_{j=1}^{d-1}(n_{k-d+j}+\tilde\A_{k+j})}{n_k}.
\end{equation}
Using telescopic sums, we can compute using (\ref{eq:mkcomputation}) that ii) holds with this choice of $\e_k$, so it remains to show that $|\e_k|\le 1/(2k)$.
If, using i), we replace $m_{k-1}$ with $(1+\de_{k-1})n_{k-1}-a_{k-1}$ in (\ref{eq:ek_defn}), then 
take the modulus of each term and use that $|\tilde\A_i|\le|\A_i|+2$ for $i>d$, 
then property \ref{item:ek_ALL_CASES} in Lemma \ref{lem:the_BETTER_sequences_exist} gives
that $|\e_k|\le 1/(4k)\le 1/(2k)$.
%it's elementary to verify that both bounds follow from property \ref{item:ek_ALL_CASES} in Lemma \ref{lem:the_BETTER_sequences_exist}.

Lastly we verify iii).
Noting the similarity to ii), if we define
\begin{equation}\label{eq:ek'_defn}
\e_k':=\e_k-\frac{\tilde\A_{k+d}}{n_k},
\end{equation}
then,
using \eqref{eq:mkcomputation} for $k+d$ instead of $k$,
a straightforward calculation shows that iii) holds with this choice of $\e_k'$.
Note that in the previous paragraph we proved 
%The attentive reader will have realized that the constant 4 in property \ref{item:ek_ALL_CASES} of Lemma \ref{lem:the_BETTER_sequences_exist} actually implies 
the stronger estimate $|\e_k|\le 1/(4k)$.
Therefore, $|\e_k'|\le 1/(2k)$ follows from the triangle inequality applied to (\ref{eq:ek'_defn}), in conjunction with property \ref{item:ek'_condition} of Lemma \ref{lem:the_BETTER_sequences_exist}.
\end{proof}

%\fix{New until here. T}

\subsection{Injectivity of $\pi_{\ell_k}:\Gamma\to{\ell_k}$}

\begin{lemma}\label{lem:intervals_are_disjoint}
Let $(\ell_k)$ be the sequence of lines given by Lemma \ref{lem:better_sequence_of_lines}, let $(a_k)$ and $(n_k)$ be the real sequences given by Lemma \ref{lem:the_BETTER_sequences_exist}, and let $\Gamma=\Gamma((\ell_k),(a_k),(n_k))$ be the corresponding set.
Then, for fixed large enough $k$, the intervals $I_{k}^{(h,j)}$ defined in \eqref{eqn:Idefn} are disjoint for all distinct pairs $(h,j)$ with $h \in \Z$, $1 \leq j \leq 2^{m_{k-1}}$.
\end{lemma}
\begin{proof}
We consider two intervals $\smash{I_k^{(h,j)}}$ and $\smash{I_k^{(h',j')}}$.
If $h<h'$, then the distance between the left endpoint of $\smash{I_k^{(h',j')}}$ and the right endpoint of $\smash{I_k^{(h,j)}}$ is
\begin{align*}
(h'-h)2^{-n_k+a_k}+(j'-j)2^{-n_k+1}-2^{-n_k}
&\geq 2^{-n_k+a_k}-2^{-n_k+1+m_{k-1}}-2^{-n_k}.
\end{align*}
Next, property \ref{item:inj_condition} of Lemma \ref{lem:the_BETTER_sequences_exist}, in conjunction with Lemma \ref{lem:mkcomputation} i), imply that that $a_k > 2+m_{k-1}$ for $k$ large enough that $|\delta_{k-1}|<1$.
Therefore the above distance is at least
\begin{align*}
2^{-n_k+2+m_{k-1}}-2^{-n_k+1+m_{k-1}}-2^{-n_k}
%&\geq 2^{-n_k+1+m_{k-1}}-2^{-n_k}
\geq 2^{-n_k+1}-2^{-n_k}
= 2^{-n_k}.
\end{align*}
With a positive separating distance, the intervals are disjoint. 
On the other hand, if $h=h'$ then we may assume $j'>j$, so the distance between the left endpoint of $\smash{I_k^{(h',j')}}$ and the right endpoint of $\smash{I_k^{(h,j)}}$ is
\begin{align*}
(j'-j)2^{-n_k+1}-2^{-n_k}
&\ge 2^{-n_k+1}-2^{-n_k}
= 2^{-n_k},
\end{align*}
hence the intervals are disjoint in this case as well.
\end{proof}
%\begin{proof}
%We consider two intervals $\smash{I_k^{(h,j)}}$ and $\smash{I_k^{(h',j')}}$, and recall that they have width $2^{-n_k}$. Then if $h\neq h'$ we have %that the distance between the intervals must be at least \fix{why this first term?}
%``assuming h'>h, the dist between left and right endpt... is''
%\begin{align*}
%2^{-n_k+a_k}-|j-j'|2^{-n_k+1}-2^{-n_k}&\geq2^{-n_k+a_k}-2^{-n_k+1+m_{k-1}}-2^{-n_k}\\
%&\geq 2^{-n_k+2+m_{k-1}}-2^{-n_k+1+m_{k-1}}-2^{-n_k}\\
%&\geq 2^{-n_k+1+m_{k-1}}-2^{-n_k}\\
%&\geq 2^{-n_k+1}-2^{-n_k}\\
%&\geq 2^{-n_k}\\
%\end{align*}
%where we have used Property \ref{item:inj_condition} \fix{this property needs to be replaced with one that IMPLIES the estimate on $m_k$, but is %only stated in terms of $n_k$ and $a_k$} of Lemma \ref{lem:the_BETTER_sequences_exist}, that $a_k > 2+m_{k-1}$. With a positive separating distance %the intervals are disjoint. On the other hand, if $h=h'$ then necessarily $j \neq j'$ and the distance between the intervals is at least
%\begin{align*}
%|j-j'|2^{-n_k+1}-2^{-n_k}&>2^{-n_k+1}-2^{-n_k}\\
%&> 2^{-n_k}\\
%\end{align*}
%and they are again disjoint.
%\end{proof}
Having proved that the intervals $\smash{I_k^{(h,j)}}$ are disjoint, we may proceed to injectivity.
\begin{lemma}\label{lem:disjoint_projection}
The map $\pi_{\ell_k}:\Gamma\to\ell_k$ is injective.
\end{lemma}
\begin{proof}
Suppose we have two points $x,y \in \Gamma$ with $\pi_{\ell_k}(x) = \pi_{\ell_k}(y)$, and take the subsequence $\{\ell_{k_i}\}$ which is identically $\ell_k$ (here we use property \ref{item:infinite_lines} of Lemma \ref{lem:better_sequence_of_lines}). 
Then the point $\pi_{\ell_k}(x) = \pi_{\ell_k}(y)$ on $\ell_k$ is contained in a sequence of intervals $\smash{I_{k_i}^{(h_i,j_i)}}$. 
Since we proved above that these intervals are disjoint for large enough $k_i$, the choice of $(h_i, j_i)$ is unique. 
Then we have $x, y \in \pi_{\ell_k}^{-1}(I_{k_i}^{(h_i,j_i)})\cap R_{{k}-1}^{(j_i)}$ since (by construction of $\Gamma_{k_i}$) this is the only parallelotope whose image under $\pi_{\ell_k}$ is $\smash{I_{k_i}^{(h_i,j_i)}}$.
Finally we check that
\[
\lim_{k \to \infty}\diam(R_{k}^{j}) = 0.
\]
First, recognize that $\diam(R_{k}^{j})$ is constant across $j$ by construction. Recall those vectors $v_i$ and angles $\phi_i$ used in the proof of Lemma \ref{l:intervallength} and the equations \eqref{e:cosphi} and \eqref{eqn:alpha_beta_defn} relating them to each other and the constants $\beta_i$. Then using property \ref{item:limiting_diam_assumption} of Lemma \ref{lem:the_BETTER_sequences_exist} and that $d \geq 2$ we have
\[\diam(R_{k}^{j}) \le
\sum_{i=k-d+1}^{k} |v_i| =
\sum_{i=k-d+1}^{k}\frac{2^{-n_i}}{\cos\phi_i}=
\sum_{i=k-d+1}^{k}\frac{2^{-n_i+\beta_i}}{2(d-1)}\le
d\frac{2^{-k}}{2(d-1)}\le
2^{-k}
 \]
which indeed tends to $0$ as $k \to \infty$. Showing that the diameters of the paralleletopes tends to $0$ is enough to finish the proof because then $x=y$, so $\pi_{\ell_k}$ is indeed injective.
\end{proof}

%%%%%%%%%%%%%%%%%%%%%%%%%%%
%%%%%%%%%%%%%%%%%%%%%%%%%%%
%%%%%%%%%%%%%%%%%%%%%%%%%%%
%%%%%%%%%%%%%%%%%%%%%%%%%%%
\section{Dimension and measure computations for $\Gamma$}\label{sec:dimensioncomp}

Fix a line $\ell\in\sL$.
In this section we prove the three estimates $\dim \pi_\ell(\Gamma)\leq t$, $\dim \Gamma\leq d-1+\dim\pi_\ell(\Gamma)$, and $\dim\Gamma\geq d-1+t$ in Subsections \ref{subsec:ubdimproj}, \ref{subsec:UBdimGamma}, and \ref{subsec:MDP} through \ref{subsec:MDPcase2}, respectively.
Together these clearly imply the first paragraph of Theorem \ref{thm:linesubspaces}.

Additionally, we show in Subsections \ref{subsec:MDPcase1} and \ref{subsec:MDPcase2} that for $t\in[0,1)$, the set $\Gamma$ has positive $(d-1+t)$-capacity provided $a_k/n_k$ satisfies the following estimate for sufficiently large $k$,
%\fix{$d-s$ is changed back to $1-t$ in \eqref{eq:MDPconvergencecondition} and \eqref{eq:measureconvergencecondition}. -T}
\begin{equation}\label{eq:MDPconvergencecondition}
    1-t=\lim_{i\to\infty}\frac{a_i}{n_i}\geq \frac{a_k}{n_k}+\frac{1}{k};
\end{equation}
this is option (1) in Theorem \ref{thm:linesubspaces}, as well as property \ref{item:pm_limit} of Lemma \ref{lem:the_BETTER_sequences_exist}.
Separately, we will argue in Subsection \ref{subsec:thm1option2} that for $t\in(0,1]$, the $t$-dimensional Hausdorff measure of $\pi_\ell(\Gamma)$ is zero provided that for sufficiently large $k$, the ratio $a_k/n_k$ satisfies the following inequality,
\begin{equation}\label{eq:measureconvergencecondition}
    1-t=\lim_{i\to\infty}\frac{a_{i}}{n_{i}}
\leq\frac{a_k}{n_k}-\frac{1}{k},
\end{equation}
which is option (2) in Theorem \ref{thm:linesubspaces}, as well as property \ref{item:pm_limit} of Lemma \ref{lem:the_BETTER_sequences_exist}.
Observe that these conditions are not compatible, hence for $t\in(0,1)$ we cannot guarantee both (1) and (2) in Theorem \ref{thm:linesubspaces}.

%\fix{This paragraph is new. -T}
% \fix{@Dylan: please read this paragraph that Tamas added a while ago, and help me decide what to do with it... I don't see how it fits in to the revised paper, but since Tam\'as thought it necessary to add at one point I thought I'd check with you before deleting.}
% Note that for $t=0$, we have that \eqref{eq:MDPconvergencecondition}
% implies that $n_k-a_k\ge n_k/k$. 
% Thus in the $t=0$ case $n_k-a_k$ is
% as large as we want if $n_k$ is large enough. 
% On the other hand, this clearly holds also 
% when $t>0$ since $a_k/n_k\to 1-t$. 
% Therefore in Lemma~\ref{lem:mkcomputation} it
% is enough that $n_k$ grows fast enough.

%that $\dim \pi_\ell(\Gamma)=t$, since
%Throughout this section, we take $s:=\inf_{\ell}\{t_\ell+d-1\}$.
%These are sufficient, since...
%When we restrict our attention to a single line $\ell\in\sL$, we sometimes refer to a projection $\pi:\Gamma\to\ell$ rather than $\pi_\ell:\Gamma\to\ell$ for notational convenience.
%Likewise, we consider the target projection to be $t$, or $t_\ell$ when more specificity is needed.

\subsection{The upper bound $\dim\pi_\ell(\Gamma)\leq t$}\label{subsec:ubdimproj}

It suffices to construct a sequence of finite covers $\set{U_i}$ for $\pi_\ell(\Gamma)$ such that for every $\e,\e'>0$, for sufficiently large $i$, we have
\[\sum_j\paren{\diam\, U_i^{(j)}}^{t+\e}<\e'.\]
We examine the natural sequence of coverings generated by our construction.
Namely, there exists a subsequence $\{\ell_{k_i}\}$ of $\{\ell_k\}$ which is identically $\ell$, and as defined previously, the projection of $\smash{R_{k_i}=\cup_jR^{(j)}_{k_i}}$ into $\ell$ consists of $2^{m_{k_i}}$ intervals of width $2^{-n_{k_i}}$.
Accordingly, we define the cover $\smash{\{U_i^{(j)}:j=1,\dots,2^{m_{k_i}}\}}$ to be the collection of these intervals.
%Recall that by (\ref{eq:limit_ak/nk}) we have $t=1-\lim a_{k_i}/n_{k_i}$.

It follows that the above sum is $2^{m_{k_i}-n_{k_i}(t+\e)}$.
By Lemma \ref{lem:the_BETTER_sequences_exist} property \ref{item:limit_ak/nk} and Lemma \ref{lem:mkcomputation} i) we see as $i \to \infty$,
\begin{align*}\label{eq:projuptoshow1}
    \lim_{i \to \infty}\left[ m_{k_i}-n_{k_i}(t+\varepsilon)\right]&=\lim_{i \to \infty}\left[\paren{\lim_{j\to\infty}\frac{a_{k_j}}{n_{k_j}}}-\frac{a_{k_i}}{n_{k_i}}+\delta_{k_i}-\varepsilon\right]n_{k_i}=-\infty
\end{align*}
as needed, since $\delta_{k_i}\to0$ and $n_{k_i}\to\infty$. %arbitrarily fast. %(in particular, it is smaller than $\e$ for large $k$).

%If we impose the growth conditions $n_{k_i}\geq i^2$ and $n_{k_i}\geq i\cdot m_{k_{i-1}}$, then (\ref{eq:projuptoshow1}) holds for large $i$, as needed.  

\subsection{Option (2) in Theorem \ref{thm:linesubspaces}}\label{subsec:thm1option2}
Here we verify that for $t\in(0,1]$, if we assume (\ref{eq:measureconvergencecondition}), then we have $H^t(\pi_\ell(\Gamma))=0$. Utilizing the same sequence of covers $U_i$ defined above, we compute that
\[\sum_j\paren{\diam\, U_i^{(j)}}^{t}=2^{m_{k_i}-n_{k_i}t}.\]
Applying Lemma \ref{lem:the_BETTER_sequences_exist} property \ref{item:limit_ak/nk} and Lemma \ref{lem:mkcomputation} i) as above, as well as (\ref{eq:measureconvergencecondition}), we see that

\begin{align*}\label{eq:projuptoshow1}
    \lim_{i \to \infty} m_{k_i}-n_{k_i} t = \lim_{i \to \infty}\left[\paren{\lim_{j\to\infty}\frac{a_{k_j}}{n_{k_j}}}-\frac{a_{k_i}}{n_{k_i}}+\delta_{k_i}\right]n_{k_i}\leq \lim_{i \to \infty}\left[-\frac{1}{k_i}+\delta_{k_i}\right]n_{k_i}=-\infty,
\end{align*}
since $|\delta_{k_i}|\le 1/2k_i$ by Lemma \ref{lem:mkcomputation}, and because $n_{k_i}\ge 2k_i^2$ by property \ref{item:option_2} of Lemma \ref{lem:the_BETTER_sequences_exist}.
Hence the $t$-dimensional Hausdorff measure of $\pi_\ell(\Gamma)$ is $0$, provided (\ref{eq:measureconvergencecondition}) holds.

\subsection{The upper bound $\dim\Gamma\leq d-1+\dim\pi_\ell(\Gamma)$}\label{subsec:UBdimGamma}

This follows from the observation that $\Gamma$ is contained in some isometric image of $\ell^\perp\times\pi_\ell(\Gamma)$.

%\fix{compactify proof in this section... not as many things need to be aligned probly}

\subsection{The setup for the lower bound on the size of $\Gamma$} % $\dim\Gamma\geq d-1+t$}
\label{subsec:MDP}

To complete the proof of Theorem~\ref{thm:linesubspaces} it remains to prove $\dim\Gamma\ge d-1+t$ and, in order 
to get option (1), to show that if $t\in[0,1)$ and \eqref{eq:MDPconvergencecondition} holds
then $\Gamma$ has positive $(d-1+t)$-capacity and infinite $(d-1+t)$-dimensional Hausdorff measure. 

Towards this, we define a mass distribution on $\Gamma$ in the natural way, starting with unit mass for $\Gamma_0$, uniformly distributing the mass from each paralellepiped in $\Gamma_{k-1}$ into the smaller sub-parallelotopes in $\Gamma_{k}$, and letting $\mu$ be the limiting mass distribution.
Let $Q$ 
%$Q\sub[0,1]^d$ 
be a ball of diameter $2^{-q}$. 
%Throughout the following subsections, we take $s:=d-1+t$. 
By the mass distribution principle (see for example \cite[pp. 61]{falconer_fractal}), to prove that $\dim\Gamma\geq d-1+t$ it would suffice to 
show $\mu(Q) \le 2^{-qs}$ for every $s<d-1+t$.
%In order to show option (1) in Theorem \ref{thm:linesubspaces}, 
In option (1) we also need capacity estimates, so to make the argument more consistent for the two situations, instead of the mass distribution principle we will apply (for both options) the following slightly stronger standard result, which we prove for completeness.

%\fix{In the following lemma every $\nu$ was replaced by $\mu$ (because later we cite \eqref{eq:toshowforlemma} for $\mu$) and a $c$ was replaced by $C$. -T}
\begin{lemma}\label{lem:poscapacity}
If $s>0$ and $\mu$ is a finite Borel measure supported on a compact set $K$, and
\begin{equation}\label{eq:toshowforlemma}
    \mu(Q)\le \frac{2^{-qs}}{q^2} \quad \text{for any ball of diameter } 2^{-q} \text{ for large enough $q$}, 
\end{equation}
%, for large enough $q$, $\mu(Q)\le 2^{-qs}/q^2$ for every ball of diameter $2^{-q}$, 
then the $s$-capacity of K is positive and $K$ has infinite $s$-dimensional Hausdorff dimension.
\end{lemma}
\bpf
By the definition of $s$-capacity $C_s$ (see \cite{mattila_fractal}) in order to show $C_s(K)>0$ it is enough prove that $I_s(\mu)<\infty$, where
$I_s(\mu)=\int\int |x-y|^{-s} d\mu(y) d\mu(x)$ is the $s$-energy of $\mu$. As in \cite{mattila_fractal}, the inner integral can be rewritten as
$$
\int |x-y|^{-s} d\mu(y) = s\int_0^{\infty} r^{-s-1} \mu(B(x,r)) dr,
$$
where $B(x,r)$ denotes the ball centered at $x$ with radius $r$. 
Since $\mu$ is a finite measure, this shows that in order to prove that $I_s(\mu)$ is finite it is enough to prove that for some fixed $r_0$ and $C$ 
(not depending on $x$) we 
have
$$
\int_0^{r_0} r^{-s-1} \mu(B(x,r)) dr \le C.
$$
Applying the assumption of the lemma for $q=-\log_2(2r)$ and taking $r_0$ small enough, we get that $\mu(B(x,r))\le (2r)^s / (\log_2(2r))^2$ for $0<r<r_0$,
which implies that the above inequality indeed holds for some finite constant $C$, which does not depend on $x$. 
%Therefore
%$$
%\int |x-y|^{-s} d\mu(y) \le s\int_0^{\diam K} \frac{2^s}{(\log_2 (2r))^2 } dr < \infty.
%$$
%Since this estimate does not depend on $x$, and $\mu$ is a finite measure, this implies that $I_s(\mu)$ is indeed finite and so $K$ has positive $s$-capacity. 

%\fix{@Tam\'as: csinalni. It would be useful if the last sentence were the following:}.

Finally, by \cite[Theorem 8.7 (1)]{mattila_fractal}, we have that if $K$ has positive $s$-capacity then it also has infinite $s$-dimensional Hausdorff dimension, as needed.
\epf

By the above lemma, it remains to show the following.
\begin{claim}\label{claim:final}
(i) If $t\in[0,1)$ and \eqref{eq:MDPconvergencecondition} holds then we have \eqref{eq:toshowforlemma} for $s=d-1+t$.
(ii) If $t\in(0,1]$ and we assume only $\lim a_i/n_i=1-t$ then \eqref{eq:toshowforlemma} holds for every $s\in [d-1,d-1+t)$.
\end{claim}

To prove this claim, we consider two cases which together cover all possible values of $q$: namely, either $2^{-n_{k+1}+a_{k+1}}\leq 2^{-q}<2^{-n_{k}}$, or $2^{-n_{k}}\leq 2^{-q}<2^{-n_{k}+a_{k}}$ for some uniquely chosen index $k$.
It is clear that these cover all possible cases, because property \ref{item:furnishes_MDP_cases} of Lemma \ref{lem:the_BETTER_sequences_exist} implies that $-n_{k+1}+a_{k+1}\leq -n_{k} < -n_{k}+a_{k}$.

\subsection{Case 1: $2^{-n_{k+1}+a_{k+1}}\leq 2^{-q}<2^{-n_{k}}$}\label{subsec:MDPcase1}

Here, the diameter of $Q$ is greater than the length of the shortest translation vector between two $\smash{R_{k+1}^{(i)}}$, but small enough that a translated copy fits inside the containing $\smash{R_{k}^{(j')}}$.
This is illustrated in Figure \ref{fig:case1}.
\begin{figure}[!ht]
  \centering
  \includegraphics[width=0.8\textwidth]{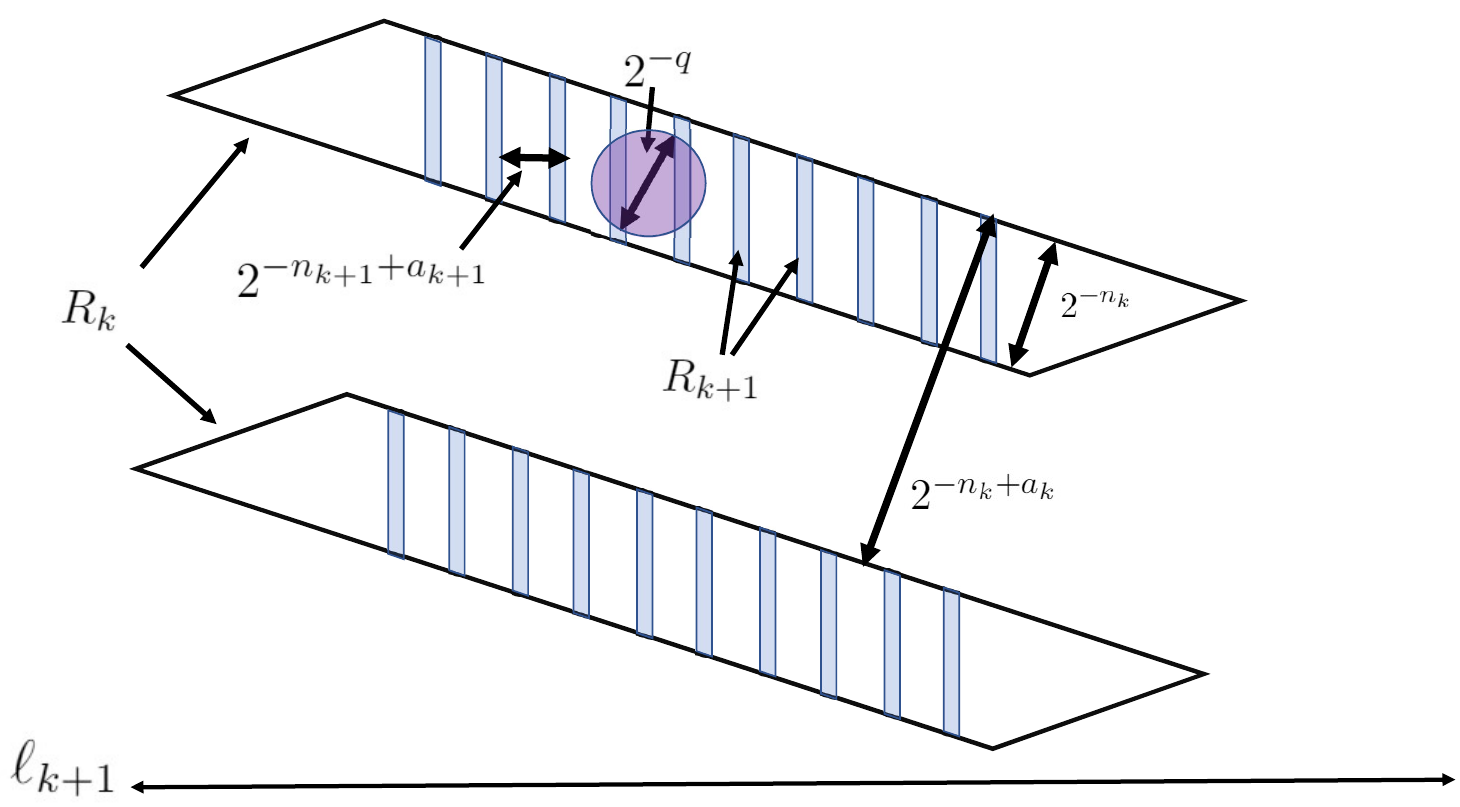}
  \caption{Positioning of $Q$ in Case 1}
  \label{fig:case1}
\end{figure}
In this case, we first obtain the following basic estimate,
\begin{align*}
    \mu(Q)
    &\leq
    \mu(R_{k+d}^{(i)})
    \cdot
    \#\set{R_{k+d}^{(i)}:R_{k+d}^{(i)}\cap Q\neq\emptyset}.
\end{align*}
By our construction the mass of each $R_{k+d}^{(i)}$ is $2^{-m_{k+d}}$, and the second factor can be bounded as follows,
\begin{align*}
\#\set{R_{k+d}^{(i)}:R_{k+d}^{(i)}\cap Q\neq\emptyset}
&\leq\prod_{j=1}^d\max_{i}\Bigg(\#\set{R_{k+j}^{(i')}\in R_{k+j-1}^{(i)}:R_{k+j}^{(i')}\cap Q\neq\emptyset}\Bigg)\\
&\leq\prod_{j=1}^{d}\left\lceil
\frac{2^{-q}}{2^{-n_{k+j}+a_{k+j}}}\right\rceil\\
&\leq\prod_{j=1}^{d}
\frac{2^{-q+1}}{2^{-n_{k+j}+a_{k+j}}},
\end{align*}
where the first estimate holds since $Q$ can intersect
only one $R_k^{(i)}$ and the second estimate holds because the shortest translation vector between any two $R_{k+j}^{(i)}$ has length $2^{-n_{k+j}+a_{k+j}}$ by our construction and all such sets must intersect $Q$. The final estimate holds by the fact that $\lceil x \rceil \le 2x$ for any $x \ge 1$, which holds here by using the case hypothesis and that the sequence $n_k-a_k$ is non-decreasing by property \ref{item:furnishes_MDP_cases} of Lemma \ref{lem:the_BETTER_sequences_exist}.

Hence, $(\ref{eq:toshowforlemma})$ is implied by the following,
\[-m_{k+d}+\sum_{j=1}^d(n_{k+j}-a_{k+j})+d\leq q(d-s)-2\log_2 q.\]
By Lemma \ref{lem:mkcomputation} ii) this is equivalent to 
\[a_k+\e_kn_k\leq q(d-s)-2\log_2 q.\]
%where $\e_k\to0$ as quickly as we want by taking $n_k$ large enough.
Because $d>s$ both in (i) and (ii), there exists $K_1$ so that for $q>K_1$ we have $q(d-s)-2\log_2q$ is monotonically increasing in $q$.
Since $n_k<q$ by the hypothesis of this case, we find that it is enough to prove
\begin{equation}\label{eq:finaltoshowMDPcase1}
    s\leq d-\frac{a_k}{n_k}-\frac{2\log_2n_k}{n_k}-\e_k.
\end{equation}

%We now consider the cases $t\in[0,1)$ and $t\in(0,1]$ separately, corresponding to options (1) and (2) in Theorem \ref{thm:linesubspaces}, respectively.
To check (i) observe that 
%In the former case, 
if we assume (\ref{eq:MDPconvergencecondition})
then (\ref{eq:finaltoshowMDPcase1}) 
for $s=d-1+t$ 
is implied by 
%$n_k\geq 2k\log_2 n_k/(1-k\e_k)$, 
$2\log_2 n_k / n_k + \varepsilon_k \le 1/k$,
and this last inequality holds by property \ref{item:case1_MDP}, in conjunction with the estimate $|\e_k|<1/2k$.
To check (ii) note that the right-hand side of \eqref{eq:finaltoshowMDPcase1} tends to $d-1+t$, so \eqref{eq:finaltoshowMDPcase1} indeed holds
for large enough $k$
for any $s\in[d-1,d-1+t)$.

\subsection{Case 2: $2^{-n_{k}}\leq 2^{-q}<2^{-n_{k}+a_{k}}$}\label{subsec:MDPcase2}

Here, the diameter of $Q$ is greater than the width of an $\smash{R_{k}^{(j')}}$ projected onto $\ell_{k}$, but smaller than the distance of the shortest translation vector between two $\smash{R_{k}^{(j')}}$.
This is illustrated in Figure \ref{fig:case2}.
\begin{figure}[!ht]
  \centering
  \includegraphics[width=0.8\textwidth]{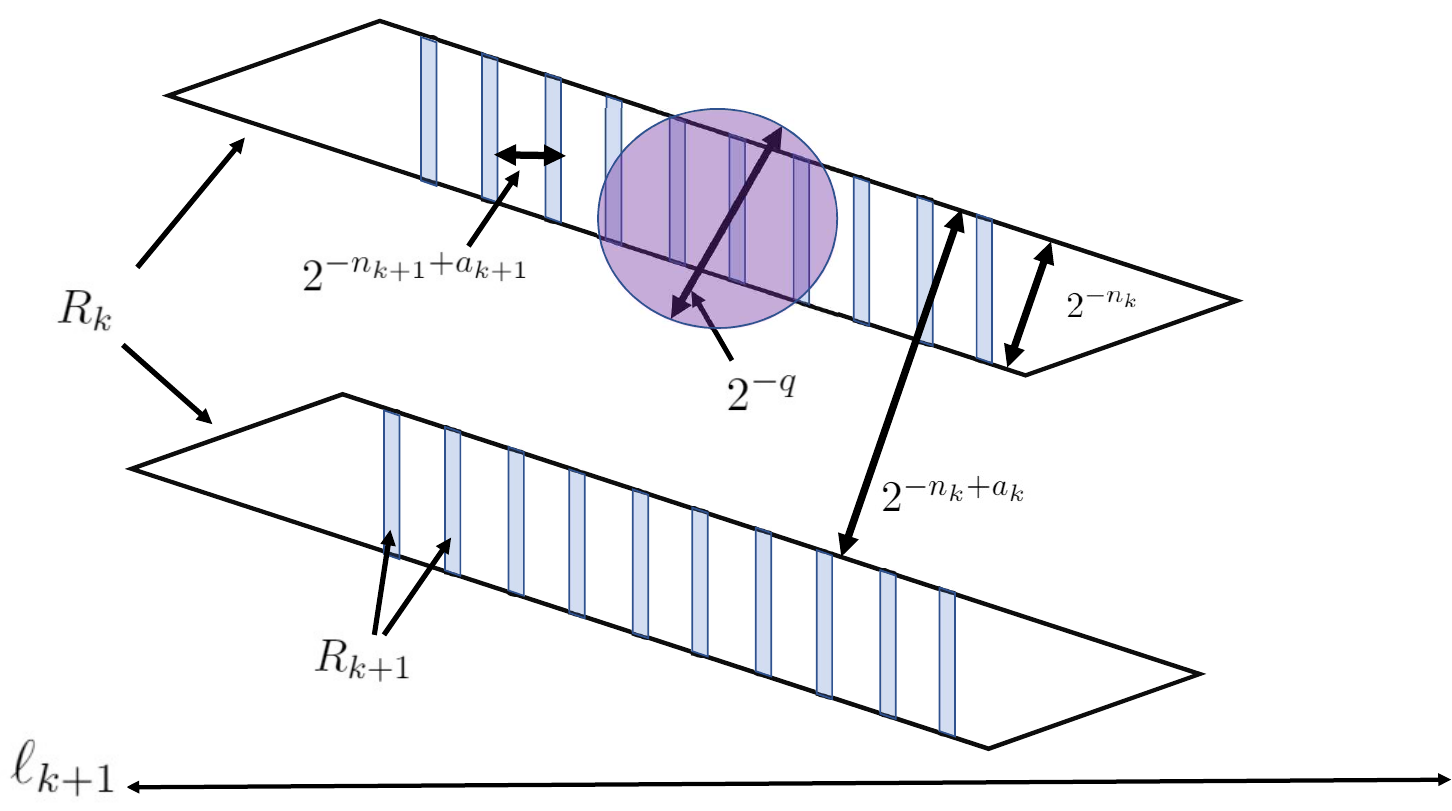}
  \caption{Positioning of $Q$ in Case 2}
  \label{fig:case2}
\end{figure}
Accordingly, this time we start with a similar basic estimate,
\begin{align*}
    \mu(Q)
    &\leq
    \mu(R_{k+d-1}^{(i)})
    \cdot
    \#\set{R_{k+d-1}^{(i)}:R_{k+d-1}^{(i)}\cap Q\neq\emptyset}.
\end{align*}
The number of $R_{k+d-1}^{(i)}$ which intersect $Q$ is bounded similarly, but this time we only take the product over the first $d-1$ terms, 
and we use that $Q$ can intersect at most two $R_k^{(i)}$.
\begin{align*}
\#\set{R_{k+d-1}^{(i)}:R_{k+d-1}^{(i)}\cap Q\neq\emptyset}
&\leq 2 \prod_{j=1}^{d-1}\max_{i}\Bigg(\#\set{R_{k+j}^{(i')}\in R_{k+j-1}^{(i)}:R^{(i')}_{k+j}\cap  Q\neq \emptyset}\Bigg)\\
&\leq 2 \prod_{j=1}^{d-1}
\left\lceil\frac{2^{-q}}{2^{-n_{k+j}+a_{k+j}}}\right\rceil\\
&\leq 2 \prod_{j=1}^{d-1}\frac{2^{-q+1}}{2^{-n_{k+j}+a_{k+j}}} \\
&= 2^d \prod_{j=1}^{d-1}\frac{2^{-q}}{2^{-n_{k+j}+a_{k+j}}}.
\end{align*}
%Again, in order to satisfy the mass distribution principle with the claimed dimension $s$, 
Hence \eqref{eq:toshowforlemma}
% it suffices to show $\mu(Q)\leq 2^{-qs}/q^2$, which 
is implied by the following,
\[-m_{k+d-1}+\sum_{j=1}^{d-1}(n_{k+j}-a_{k+j})+d
\leq q(d-s-1)-2\log_2q.\] 
%\[-m_{k+d-1}+\sum_{j=1}^{d-1}(n_{k+j}-a_{k+j})+(d-1)
%\leq q(d-s-1)-2\log_2q.\] 
By Lemma \ref{lem:mkcomputation} iii), it suffices to show
\[a_k+(-1+\e_k')n_k\leq q(d-s-1)-2\log_2q.\]
Notice that $d-s-1\leq 0$, so by the hypothesis of this case $q\leq n_k$, it suffices to show 
\begin{equation}\label{eq:toshowMDPcase2}
    s\leq d-\frac{a_k}{n_k}-\frac{2\log_2n_k}{n_k}-\e_k'.
\end{equation}
%To control the error term we impose $n_k\geq k\cdot(\psi_{k}+2\log_2n_k)$.
Note that this estimate is nearly identical to (\ref{eq:finaltoshowMDPcase1}), hence the remainder of this argument follows \ita{mutatis mutandis}.
%Arguing as in Case 1, it is easy to see that for $t\in[0,1)$ we have (\ref{eq:toshowMDPcase2}) provided we impose (\ref{eq:MDPconvergencecondition}). Identical to the computation shown in Section \ref{subsec:MDPcase1}, for $t\in(0,1]$ we may still show (\ref{eq:MDP_toshow_v2}) for each $s_n':=s-\frac{1}{n}$. This completes the argument.

%\vspace{-12mm}

%%%%%%%%%%%%%%%%%%%%%%%%%%%%%%%%%%%%%%%%%%%%%%%%%%%%%%%%%%%%
%%%%%%%%%%%%%%%%%%%%%%%%%%%%%%%%%%%%%%%%%%%%%%%%%%%%%%%%%%%%
%%%%%%%%%%%%%%%%%%%%%%%%%%%%%%%%%%%%%%%%%%%%%%%%%%%%%%%%%%%%
%%%%%%%%%%%%%%%%%%%%%%%%%%%%%%%%%%%%%%%%%%%%%%%%%%%%%%%%%%%%
\section*{Acknowledgements}
The authors would like to thank Rich\'ard Balka, Korn\'elia H\'era, Andr\'as Math\'e, and Pertti Mattila for their suggestions and remarks, as well as the referees for their detailed suggestions regarding the exposition.
%\fix{I did not like that Mattila was separated, ha also gave us suggestions. If we want to be more specific, we migh mention that he suggested using capacity and that lemma for getting option (1). Or we may put even Andras Mathe into the first sentence without specifying his help. -T}
%as well as Pertti Mattila for fruitful discussions.
%They also credit Andr\'as Math\'e for suggesting the construction of the injective function.
They would also like to thank the Budapest Semesters in Mathematics program for providing the framework under which this research was conducted.

%\printbibliography

%\fixed{Replace journal names with standard abbreviations}

%\fix{I updated the references and made some minor corrections, for example there should be no comma between the volume number and the year, and the books should be cited a bit differently. -T}

\end{document}